\documentclass[reqno]{amsart}

\usepackage{mathtools,amsmath, amsthm, amsfonts, hyperref, verbatim}
\usepackage{graphicx, subcaption}
\usepackage[labelfont=small,textfont=footnotesize]{caption}
\usepackage{wrapfig}

\newtheorem{theorem}{Theorem}[section]
\newtheorem{prop}[theorem]{Proposition}
\newtheorem{lemma}[theorem]{Lemma}
\newtheorem{corollary}[theorem]{Corollary}

\DeclareMathOperator*{\argmin}{arg\,min}

\newcommand{\R}{{\mathbb R}}
\newcommand{\Z}{{\mathbb Z}}
\renewcommand{\v}{{\mathbf v}}
\newcommand{\w}{{\mathbf w}}
\renewcommand{\u}{{\mathbf u}}
\renewcommand{\l}{{\ell}}

\usepackage{mathtools}
\usepackage{todonotes}
\DeclarePairedDelimiter{\ceil}{\lceil}{\rceil}

\begin{document}

\title{Approximating geodesics via random points}
\author{Erik Davis and
Sunder Sethuraman}
\date{}

\address{\noindent Department of Mathematics, University of Arizona,
  Tucson, AZ  85721
\newline
e-mail:  \rm \texttt{edavis@math.arizona.edu}
}

\address{\noindent Department of Mathematics, University of Arizona,
  Tucson, AZ  85721
\newline
e-mail:  \rm \texttt{sethuram@math.arizona.edu}
}

\begin{abstract}
  Given a `cost' functional $F$ on paths $\gamma$ in a domain $D\subset\mathbb{R}^d$, in the form $F(\gamma) = \int_0^1 f(\gamma(t),\dot\gamma(t))dt$, it is of interest
to approximate its 
  minimum cost and geodesic paths. 
  Let $X_1,\ldots, X_n$ be points drawn independently from $D$ according to a distribution with a density.   Form a random geometric graph on the points where $X_i$ and $X_j$ are connected when $0<|X_i - X_j|<\epsilon$, and the length scale $\epsilon=\epsilon_n$ vanishes at a suitable rate.   
	
	For a general class of functionals $F$, associated to Finsler and other distances on $D$, using a probabilistic form of Gamma convergence, we show
	that the minimum costs and geodesic paths, with respect to types of approximating discrete `cost' functionals, built from the random geometric graph, converge almost surely in various senses to those corresponding to the continuum cost $F$, as the number of sample points diverges.  In particular, the geodesic path convergence shown appears to be among the first results of its kind.
\end{abstract}



\maketitle

\section{Introduction}
Understanding the `shortest' or geodesic paths between points in a medium
is an intrinsic concern in diverse applied problems, from `optimal routing' in networks and disordered materials
to `identifying manifold structure in large data sets', as well as in studies of probabilistic $\Z^d$-percolation models, since the seminal paper of  
\cite{hammersley}
(cf. recent survey \cite{Auffinger}).  See also \cite{Hirsch}, \cite{Howard_Newman}, \cite{Howard}, \cite{Hwang}, \cite{lagatta0}, \cite{lagatta} which consider percolation in $\mathbb{R}^d$ continuum settings.

There are sometimes abstract formulas for the geodesics, from the calculus of variations, or other differential equation approaches. 
For instance, with respect to a patch of a Riemannian manifold $(M,g)$, with $M\subset \R^d$ and tensor field $g(\cdot)$, it is known that the distance function $U(\cdot)= d(x,\cdot)$, for fixed $x$, is a viscosity solution of the  Eikonal equation $\|\nabla U(y)\|_{g(y)^{-1}}=1$ for $y\neq x$, with boundary condition $U(x)=0$.  Here, $\|v\|_{A}=\sqrt{\langle v,Av\rangle}$, where $\langle\cdot,\cdot\rangle$ is the standard innerproduct on $\R^d$.
Then, a geodesic $\gamma$ connecting $x$ and $z$ may be recovered from $U$ by solving a `descent' equation, $\dot\gamma(t) = -\eta(t)g^{-1}(\gamma(t))\nabla U(\gamma(t))$, where $\eta(t)$ is a scalar function controlling the speed.  

On the other hand, computing numerically the distances and geodesics may be a complicated issue.
One of the standard approaches is the `fast marching method'
to
approximate the distance $U$,
by solving the Eikonal equation on a regular grid of $n$ points.
This method has been extended in a variety of ways, including with respect to
triangulated domains, 
as well as 
irregular samples  $\{x_1,\ldots,x_n\}$ of an Euclidean submanifold (cf. \cite{sethian1999fast}, \cite{memoli2005distance}). See also \cite{peyre2010geodesic} in the above contexts for a review. 

Alternatively,
variants of Dijkstra's or `heat flow' methods, on graphs approximating the space are sometimes used.  In Dijkstra's algorithm, distances and shortest paths are found by successively computing optimal routes to nearest-neighbor edges.  In `heat flow' methods, geodesic distances can be found in terms of the small time asymptotics of a heat kernel on the space.  
For instance, see \cite{Cabello_unitdisks}, \cite{Crane_geodesics_in_heat}, \cite{Giesen_Wagner}, \cite{Glickenstein}, \cite{Yu_geodesics_computation}.

Another
idea
has been to collect a random sample $\mathcal{X}_n$ of $n$ points from a manifold embedded in $\R^d$, put a network structure on these points, say in terms of a $\epsilon$-random geometric or $k$-nearest neighbor graph, and then approximate the `continuum' geodesics 
lengths 
by 
lengths of 
`discrete' geodesic paths found in this network. 
Presumably, under assumptions on how the points are sampled and how the random graphs are formed, as the number of points diverge, these `discrete' distances should converge almost surely to the `continuum' shortest path lengths.  Such a statistical consistency result is fundamental in `manifold learning' \cite{bern}.  For instance, the popular ISOMAP procedure \cite{tenenbaum}, \cite{bernstein2000graph} is based on these notions to elicit manifold structure in data sets.

	More specifically, let $D$ be a subset of $\R^d$ corresponding to a patch of the manifold, and consider a `kernel' $f(x,v):D\times \mathbb{R}^d \rightarrow [0,\infty)$.  Define the $f$-cost of a path $\gamma(t):[0,1]\rightarrow D$ from $\gamma(0)=a$ to $\gamma(1)=b$ as $F(\gamma)=\int_0^1f(\gamma(t),\dot\gamma(t))dt$.  The $f$-distance from $a$ to $b$ is then the infimum of such costs over paths $\gamma$.  For example, if $f(x,v) = |v|^p$, the $f$-distance is $|b-a|^p$, the $p$th power of the Euclidean distance.

With respect to a class of functions $f$ and samples drawn from a distribution on the $D$ with density $\rho$, 
papers \cite{bernstein2000graph}, 
\cite{Sajama}, and 
\cite{alamgir2012shortest} address, among other results, how $\epsilon=\epsilon_n$ and $k=k_n$ should decrease and increase respectively so that various concentration type bounds between types of discrete and continuum optimal distances hold with high probability, leading to consistent estimates.  

For instance, in 
\cite{Sajama}, for $\epsilon_n$-random graphs and smooth $\rho$, certain density dependent estimators of continuum distances were considered, where $f(x,v) = h(\rho(x))|v|$ and $h(y)$ is decreasing, smooth, constant for $|y|$ small, and bounded away from $0$.  This work extends 
\cite{bernstein2000graph}, which considered $f(x,v) = |v|$ and uniformly distributed samples.
On the other hand, in 
\cite{alamgir2012shortest}, among other results, on $k_n$-nearest-neighbor graphs, continuum distances, where $f(x,v)= h(\rho(x))|v|$ and $h$ is increasing, 
Lipschitz, 
and bounded away from $0$, were approximated (see also \cite{Hashimoto}).

  In these contexts, the purpose of this article is twofold.  First, we identify a general class of $f$-distances for which different associated discrete distances, formed from random $\epsilon_n$-random geometric graphs on a domain $D\subset \R^d$, converge almost surely to them.  
		Second, we describe when the associated discrete geodesic paths converge almost surely, in uniform and Hausdorff norms, to continuum $f$-distance geodesic paths, a type of consistency which appears to be among the first contributions of this kind.
		The main results are Theorems \ref{theorem1}, \ref{theorem2}, \ref{p>1approxthm}, \ref{p=1approxthm}, \ref{linear_p=1approxthm}, and Corollary \ref{hausdorffcor}.
			   See Section \ref{results} and Subsection \ref{remarks} for precise statements and related remarks.

		We consider the following three different discrete costs.  The first, $d_1$, optimizes on paths $\gamma$, starting and ending at $a$ and $b$ respectively, linearly interpolated between points in $\mathcal{X}_n \cup \{a,b\}$, where consecutive points are within $\epsilon_n$ of each other, and the time to traverse each link is the same.  The second, $d_2$, optimizes with respect to `quasinormal' interpolations between the points, using however the $f$-geodesic paths.  The third, $d_3$, does not interpolate at all, and optimizes a `Riemann sum' cost $\frac{1}{m}\sum_{i=0}^{m-1} f(v_i,m(v_{i+1}-v_i))$ where $m$ is the number of edges in the discrete path $\{v_0,\ldots, v_{m}\} \subset \mathcal{X}_n\cup\{a,b\}$.  We note, discrete distances $d_2$ and $d_3$, in the setting $f(x,v)= |v|$ were introduced in \cite{bernstein2000graph}, and density dependent versions were used in the results in \cite{alamgir2012shortest} and \cite{Sajama}.  The discrete distance $d_1$, although natural, seems not well considered in the literature.

	The conditions we impose on $f$ include $p$-homogeneity in $v$ for $p\geq 1$, convexity and an ellipticity condition with respect to $v$, and a smoothness assumption away from $v\neq 0$.  Such conditions include a large class of kernels $f$ associated to Finsler spaces, as well as those kernels considered in \cite{Sajama} and \cite{bernstein2000graph}, with respect to $\epsilon_n$-random graphs.  The domain $D\subset\mathbb{R}^d$ is assumed to be bounded and convex. Also, we assume that the rate of decrease of $\epsilon_n$ is such that the graph on $\mathcal{X}_n$ is connected for all large $n$.

	While a main contribution of the article is to provide a general setting in which the `discrete to continuum' convergences hold,
  we 
	remark our proof method is quite different from that in the literature, where specific features of $f$, such as $f(x,v)= |v|$ in \cite{bernstein2000graph}, are important in estimation of distances, not easily generalized.
	We give a probabilistic form of `Gamma convergence' to derive the almost sure limits, which may be of interest itself.   
	This method involves showing `liminf', `limsup' and `compactness' elements, as in the analysis context, but here on appropriate probability $1$ sets.  Part of the output of the technique, beyond giving convergence of the distances, is that it yields convergence of the minimizing discrete paths to continuum geodesics in various senses.

	The $f$-costs share different properties depending on if $p=1$ or $p>1$, and also when $d\geq 2$ or $d=1$.
	For instance, the $f$-cost is invariant to reparametrization of the path exactly when $p=1$.  Also, when $p>1$, the form of the $f$-cost may be seen to be coercive on the modulus of $\gamma$, not the case when $p=1$.  In fact, the $p=1$ case is the most troublesome, and more assumptions on $f$ and $\epsilon_n$ are required in Theorems \ref{p=1approxthm} and \ref{linear_p=1approxthm} to deal with the `linear' path cost $d_1$ and `Riemann' cost $d_3$, which are `rougher' than the `quasinormal' cost  $d_2$.
	  
	  At the same time, in $d=1$, in contrast to $d\geq 2$, all paths lie in the interval $[a,b]\subset\mathbb{R}$.  When also $p=1$, the problem is somewhat degenerate:  By invariance to reparametrization, the costs $d_1$ and $d_2$ turn out to be nonrandom and to reduce to the integral $\int_a^b f(s, 1)ds$.  Also, the cost $d_3$ is a Riemann sum which converges to this integral.

Finally, we comment on a difference in viewpoint with respect to results in continuum percolation.  The `Riemann sum' cost considered here seems related to but is different than the cost optimized in the works 
\cite{Howard_Newman}, \cite{Howard}.  There, for $p>1$, one optimizes the cost of a path $\{w_0,\ldots, w_m\}$, along random points, from the origin $0$ to $nx$, for $x\in \R^d$, given by $\sum_{i=0}^{m-1} |w_{i+1}-w_i|^p$, and infers a scaled distance $d(x)= c(d,p)|x|$, in law of large numbers scale $n$, where the proportionality constant $c(d,p)$ is not explicit. In contrast, however, in this article, given already an integral $f$-distance, the viewpoint is to optimize costs of paths of length order $1$ (not $n$ as in \cite{Howard_Newman}, \cite{Howard}), where the length scale between points is being scaled of order $\epsilon_n$, and then to recover the $f$-distance in the limit.  We note also another difference: When $f(x,v)= |v|^p$, as remarked above, the $f$-distance from the origin to $x$ is $|x|^p$, instead of $\sim |x|$ as in the continuum percolation studies.

In Section \ref{results}, the setting, assumptions and results are given with respect to three types of discrete costs.  In Section \ref{interpolating_proofs}, proofs of Theorems \ref{theorem1} and \ref{theorem2} and Corollary \ref{hausdorffcor} on the `interpolating' costs are given.  In Section \ref{riemann_proofs}, proofs of Theorems \ref{p>1approxthm}, \ref{p=1approxthm}, and \ref{linear_p=1approxthm}, with respect to `Riemann' costs, and `interpolating' costs when $p=1$, are given.
In Section \ref{appendix}, some technical results, used in the course of the main proofs, are collected.

\section{Setting and Results}
\label{results}
\begin{figure}
  \begin{subfigure}{0.48\textwidth}
    \includegraphics[width=\linewidth]{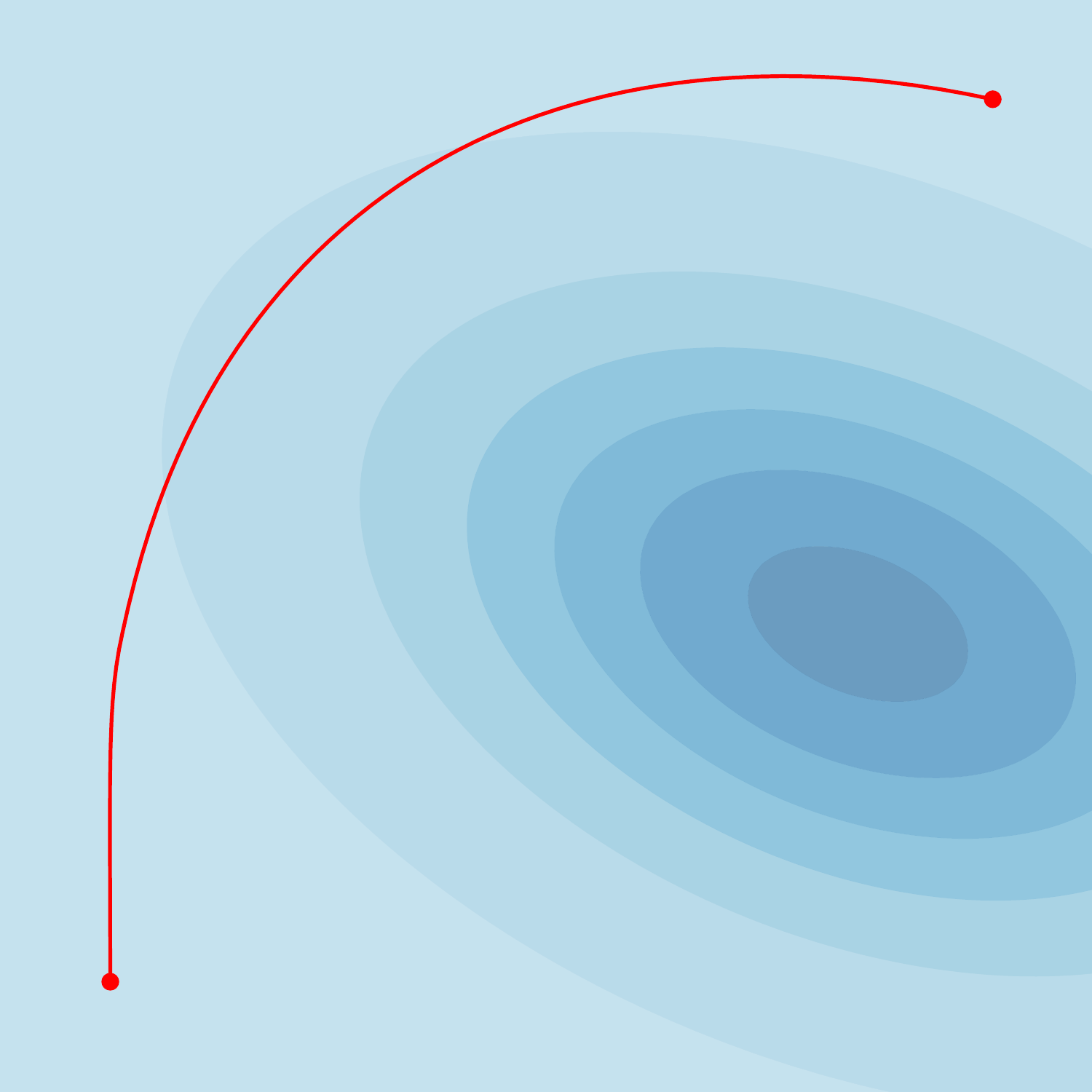}
    \caption{$F$-minimizing path, with level sets of $w$ indicated.}
    \label{fig:geodesic}
  \end{subfigure}
  \begin{subfigure}{0.48\textwidth}
    \includegraphics[trim={0.8cm 0.8cm 0.65cm 0.35cm}, width=\linewidth]{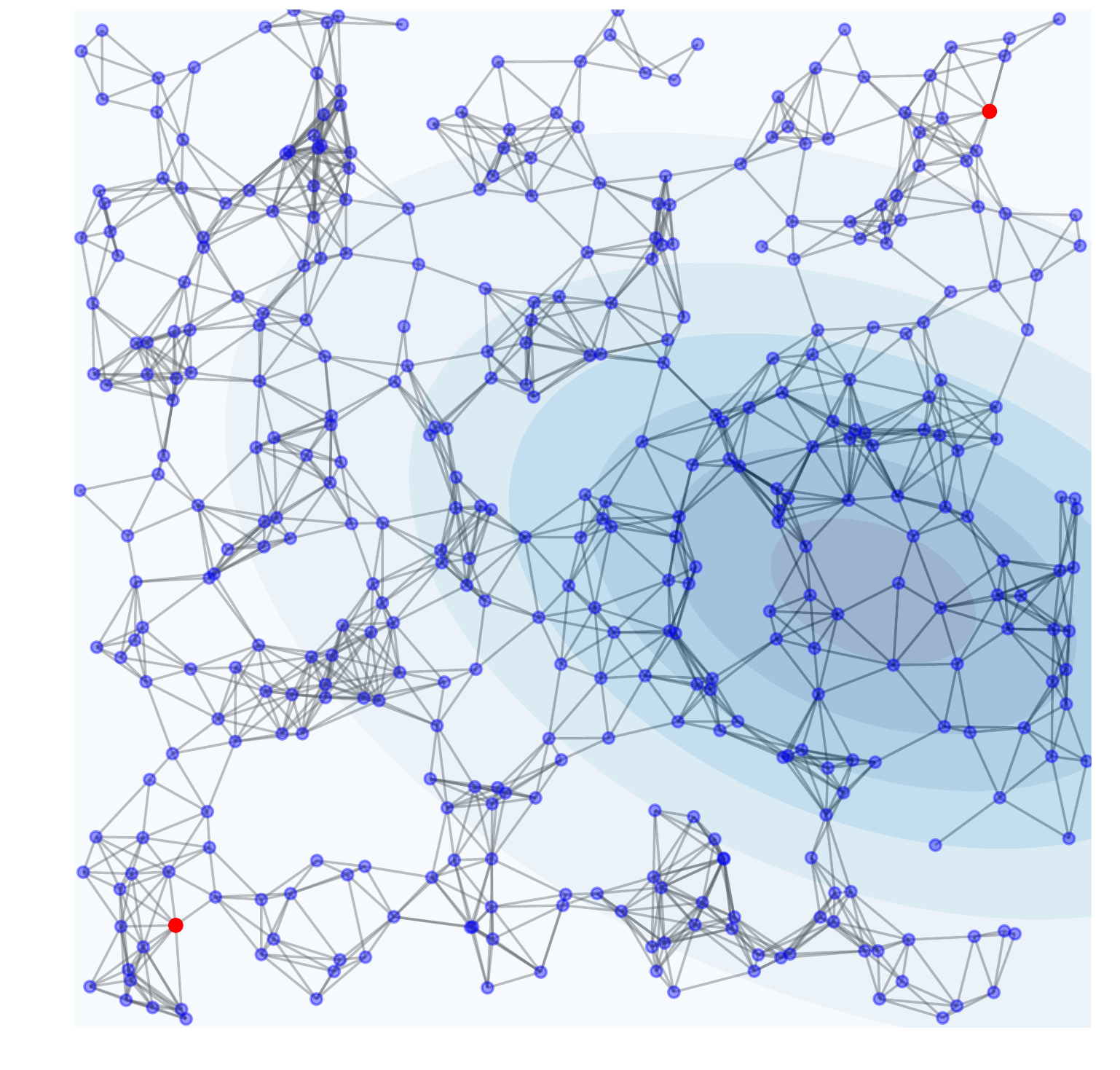}
    \caption{$\epsilon_n$-graph, on $n=400$ uniform points with $\epsilon_n = (1/400)^{0.3}$}
    \label{fig:rgg}
  \end{subfigure}
  \caption{Continuum geodesic and $\epsilon_n$-graph for $f(x,v) = w(x)|v|$ on the domain  $D = [-1,1] \times [-1,1]$, where $w(x) = 1 + 8 \exp(-2 (x_1 - 1/2)^2 + xy + 2y^2)$, and $a = (-0.8,-0.8)$, $b = (0.8,0.8)$}
  \label{fig1}\end{figure}

We now introduce the setting of the problem, and `standing assumptions', which hold throughout the article.

For $d\geq 1$, we will be working on a subset $D\subset\R^d$, 
\begin{equation}
\label{domain_assumption}
{\rm which \ is \ the \ closure \ of \ an \ open, \ bounded, \ convex \ domain.}
\end{equation} 
Therefore, $D$ is a Lipschitz domain (cf. Corollary 9.1.2 in \cite{Agronowich}, Section 1.1.8 in \cite{Mazya}).

Consider
points $a,b \in D$ and let $\Omega(a,b)$ denote the space of Lipschitz
paths $\gamma: [0,1] \to D$ with $\gamma(0) = a$ and $\gamma(1) =
b$. Given $f : D \times \mathbb{R}^d \to [0,\infty)$, we define the
cost $F : \Omega(a,b) \to [0,\infty)$ by
\begin{equation*} 
F(\gamma) = \int_0^1 f(\gamma(t),\dot{\gamma}(t)) \, dt,
\end{equation*}
and associated optimal cost
\begin{equation} \label{dfdef}
d_f(a,b) = \inf_{\gamma \in \Omega(a,b)} F(\gamma).
\end{equation}

We will make the following assumptions on the integrand $f$:
\begin{itemize}
\item[ (A0)] $f$ is continuous on $D\times \mathbb{R}^d$, and $C^1$ on
$D \times (\mathbb{R}^d \setminus \{0\})$,
\item[(A1)] $f(x,v)$ is convex in $v$,
\item[(A2)]
there exists $p \geq 1$ such that $f(x,v)$ is $p$-homogenous in $v$,
\begin{equation} \label{fhomo}
  f(x,\lambda v) = \lambda^p f(x,v) \text{ for } \lambda > 0,
\end{equation}
\item[(A3)] there exist constants $m_1,m_2 > 0$ such that
\begin{equation} \label{fbounds} m_1|v|^p \leq f(x,v) \leq m_2 |v|^p
  \text{ for all } x \in D.
\end{equation}
\end{itemize}

We remark, when $p>1$ and $p$-homogenity (A2) holds, that $f$ may be extended to a $C^1$ function on $D\times \mathbb{R}^d$. 

 Part of the reasoning for the assumptions (A0)-(A3) is that they include, for $p\geq 1$, the familiar kernel $f(x,v) = |v|^p$, for which, when $p=1$, $F(\gamma)$ is the arclength of the path $\gamma$ and $d_f(a,b)$ is the length of the line segment from $a$ to $b$.

 Also, under these assumptions on $f$, it is known that the infimum in
\eqref{dfdef} is attained at a path in $\Omega(a,b)$, perhaps nonuniquely (see
Proposition \ref{existence} of the appendix).  
In addition, we remark, when $p=1$, under additional differentiability assumptions, $d_f$ represents a Finsler distance (cf. \cite{Paiva}, \cite{tamassy} and references therein).

When $p=1$, the cost has an interesting scaling property:  By $1$-homogeneity of $f$, the cost $F$
is invariant under smooth reparameterization of
paths. That is, given a path $\gamma \in \Omega(a,b)$ and smooth, increasing
$s : [0,1] \to [0,1]$, with $s(0) = 0$ and $s(1) = 1$, one has
$F(\tilde{\gamma}) = F(\gamma)$ where
$\tilde{\gamma}(t) = \gamma(s(t))$.

This property allows to deduce, when $p=1$, that 
$d_f$ satisfies the triangle property (not guaranteed when $p>1$): Let $\gamma_1$ be a path from $u$ to $w$, and $\gamma_2$ be a path from $w$ to $z$.  Write
\begin{eqnarray}
&&\int_0^1 f(\gamma_1(t),\dot\gamma_1(t))dt +  \int_0^1 f(\gamma_2(t),\dot\gamma_2(t))dt\nonumber\\
&&\ \ \ = \ \int_0^{1/2} f(\gamma_1(2s), 2\dot\gamma_1(2s))ds + \int_0^{1/2} f(\gamma_2(2s), 2\dot\gamma_2(2s))ds\nonumber\\
&&\ \ \ = \ \int_0^1f(\gamma_3(t),\dot\gamma_3(t))dt,
\label{p=1triangle}
\end{eqnarray}
where $\gamma_3$ is a path from $u$ to $z$, following $\gamma_1(2\cdot)$ and $\gamma_2(2\cdot)$ on time intervals $[0,1/2]$ and $[1/2,1]$ respectively.
Optimizing over $\gamma_1$, $\gamma_2$ and $\gamma_3$ gives $d_f(u,w) +d_f(w,z)\geq d_f(u,z)$.

We now construct a random geometric graph on $D$ through which approximations of $d_f$ and its geodesics will be made.  
Let
$\{X_i,X_2,\ldots \}\subset D$ be a sequence of independent points, identically distributed according to a distribution $\nu$ with probability density $\rho$.  For
each $n \in \mathbb{N}$, let $\mathcal{X}_n = \{X_1,\ldots,X_n\}$ and fix a length scale $\epsilon_n > 0$. With
respect to a realization $\{X_i \}$, we define a graph
$\mathcal{G}_n(a,b)$, on the vertex set $\mathcal{X}_n \cup \{a,b\}$,
by connecting an edge between $u,v$ in $\mathcal{X}_n \cup \{a ,b \}$
iff $0<|u - v| < \epsilon_n$, where $|\cdot|$ refers to the Euclidean
distance in $\mathbb{R}^d$. 

For $u,v \in \mathcal{X}_n\cup\{a,b\}$, we say that a finite sequence
$(v_0,v_1,\ldots,v_m)$ of vertices is a {\it path} with $m$-steps
from $u$ to $v$ in $\mathcal{G}_n(a,b)$ if $v_0 = u$, $v_m = v$, and
there is an edge from $v_i$ to $v_{i+1}$ for $0 \leq i < m$. Let
$V_n(a,b)$ denote the set of paths from $a$ to $b$ in
$\mathcal{G}_n(a,b)$.

We will assume a certain decay rate on $\epsilon_n$,
namely that $\lim_{n \uparrow \infty} \epsilon_n = 0$ and
\begin{equation} \label{rateassumption}
  \limsup_{n \to \infty} \frac{(\log n)^{1/d}}{n^{1/d}} \frac{1}{\epsilon_n} = 0.
\end{equation}
Under this type of decay rate, almost surely, for all large $n$ and $a,b\in D$, points $a,b$ will be connected  by a path in the graph $\mathcal{G}_n(a,b)$, in other words, $V_n(a,b)$ will be nonempty.  Indeed, under this rate, the degree of a point in the graph will diverge to infinity. 
 See Proposition \ref{nearestneighbor} in the appendix, and remarks in Section \ref{remarks}.

We will also assume that the underlying probability density $\rho$ is uniformly bounded, that is, there exists a constant $c > 0$ such that 
\begin{equation}
\label{rho_assumption}
c\leq \rho(x) \leq c^{-1} \ \ \ {\rm  for\  all\ } x \in D.
\end{equation}

See Figure \ref{fig1}, parts (a) and (b), which depict a geodesic path with respect to a cost $F$, and an $\epsilon_n$-random graph.

\medskip
{\it `Standing assumptions'.}  To summarize, the assumptions, dimension $d\geq 1$, \eqref{domain_assumption} on $D$, items (A0)-(A3) on $f$ when $p\geq 1$, decay rate \eqref{rateassumption} on $\epsilon_n$, and density bound \eqref{rho_assumption} on $\rho$, denoted as the `standing assumptions', will hold throughout the article.
\medskip

In the next two Subsections, we present results on approximation of $d_f(a,b)$ and its geodesics with respect to two types of schemes, where approximating costs are built (1) in terms of `interpolations' of points in $V_n(a,b)$ and  also (2) in terms of `Riemann sums'.

\subsection{Interpolating costs}

We introduce two types of discrete costs based on `linear' and `quasinormal' paths.

\medskip
{\it Linear interpolations.}
With respect to a realization $\{X_i\}$, 
for $u,v \in D$, let $l_{u,v} : [0,1] \to D$ denote the
constant-speed linear path from $a$ to $b$, given by
\[ l_{u,v}(t) = (1-t)u + tv. \] 
Consider now
$\v=(v_0,v_1,\ldots,v_m) \in V_n(a,b)$. We define
$l_{\v} \in \Omega(a,b)$ to be the concatenation of the
linear segments $\{l_{v_{i-1},v_i}\}_{i=1}^m$, where each segment is traversed in the same time $1/m$.  More precisely, for
$i/m \leq t \leq (i+1)/m$, define
\begin{equation*} \label{concatenation}
  l_{\v}(t) =  l_{v_i,v_{i+1}}(mt-i),
\end{equation*}
and note that the resulting piecewise linear path is in $\Omega(a,b)$.

Define now a subset $\Omega_n^{l}(a,b)$ of $\Omega(a,b)$ by
\begin{equation*}
  \Omega_n^{l}(a,b) = \left\{ l_{\v }\middle| \v
	\in V_n(a,b) \right\},
\end{equation*}
and define the (random) discrete cost $L_n : \Omega_n^l(a,b) \to [0,\infty]$ by
\begin{equation*}
L_n(\gamma) = F(\gamma) \text{ for } \gamma \in \Omega_n^l(a,b).
\end{equation*}
In other words, $L_n$ is the restriction of $F$ to $\Omega_n^l(a,b)$, noting the $p$-homogenity of $f$, taking form
\begin{eqnarray}
\label{L_n_equation}
L_n(l_\v) &=& \sum_{i=0}^{m-1} \int_{i/m}^{(i+1)/m} f(l_{v_i,v_{i+1}}(mt-i),m(v_{i+1}-v_i))dt\nonumber\\
&=& m^{p-1}\sum_{i=0}^{m-1} \int_0^1 f(l_{v_i, v_{i+1}}(t), v_{i+1}-v_i)dt.
\end{eqnarray}

\medskip

{\it Quasinormal interpolations.} Define now a different discrete cost which may nonlinearly interpolate among points in paths of $V_n(a,b)$. 
We say that a Lipschitz path $\gamma$ is quasinormal with respect to $f$ if
there exists a $c > 0$ such that
\begin{equation*} \label{quasinormal}
f(\gamma(t),\dot{\gamma}(t)) = c \text{ for a.e. } t \in [0,1].
\end{equation*}

It is known, under the `standard assumptions' on $f$ (see Proposition \ref{existence}) that, for $u,v \in D$, there exists
a quasinormal path $\gamma : [0,1] \to D$, with $\gamma(0) = u, \gamma(1) = v$, which is optimal, 
$d_f(u,v) = \int_0^1 f(\gamma(t),\dot{\gamma}(t))\, dt$. 
For what follows, when we refer to a `quasinormal' path connecting $u$ and $v$, we mean such a fixed optimal path denoted by $\gamma_{u,v}$.

Given a path $\v=(v_0,\ldots,v_m) \in V_n(a,b)$,
let $\gamma_{\v}\in \Omega(a,b)$ denote the concatenation of
$\{\gamma_{v_{i-1},v_i}\}_{i=1}^m$, where each segment uses the same time $1/m$. More precisely,
for $i/m \leq t \leq (i+1)/m$, define
\begin{equation*} 
  \gamma_{\v}(t) =  \gamma_{v_i,v_{i+1}}(mt-i).
\end{equation*}
As with piecewise linear functions, define the
subset $\Omega_n^{\gamma}(a,b)$ of $\Omega(a,b)$ by
\begin{equation*}
  \Omega_n^{\gamma}(a,b) = \left\{ \gamma_{\v} \middle| \v
	\in V_n(a,b) \right\}.
\end{equation*}
Let $G_n : \Omega_n^{\gamma}(a,b) \to \mathbb{R}$ denote the
restriction of $F$ to $\Omega_n^{\gamma}(a,b)$.

 Then, with respect to a path
$\gamma=\gamma_{\v} \in \Omega_n^{\gamma}(a,b)$, by the
$p$-homogenity of $f$, we evaluate that
\begin{eqnarray}
 \label{G_n_equation}
  G_n(\gamma) &=& \int_0^1 f(\gamma(t),\dot{\gamma}(t))\, dt  \\
	& = & 
                  \sum_{i=1}^{m} \int_{(i-1)/m}^{i/m} f(\gamma_{v_{i-1},v_i}(mt-i),m\dot{\gamma}_{v_{i-1},v_i}(mt-i))\, dt \nonumber\\
              &=& m^{p-1}\sum_{i=1}^{m} \int_0^1 f(\gamma_{v_{i-1},v_i}(t),\dot{\gamma}_{v_{i-1},v_i}(t))\, dt \ =\    m^{p-1}\sum_{i=1}^{m} d_f(v_{i-1},v_i). \nonumber
                 \end{eqnarray}
                Further, by $p$-homogeneity of $f$ and optimality of $\{\gamma_{v_i,v_{i+1}}\}_{i=1}^m$, the segments of $\gamma = \gamma_{\v}$ are also optimal, in the sense that
                \begin{eqnarray}
							&&\int_{i/m}^{(i+1)/m} f(\gamma(t),
                  \dot{\gamma}(t))\, dt =	m^{p-1}\int_0^{1} f(\gamma_{v_i,v_{i+1}}(t), \dot \gamma_{v_i,v_{i+1}}(t))dt \nonumber\\
&&\ \ \ = \inf_{\widetilde{\gamma}}
                  m^{p-1}\int_{0}^{1} f(\widetilde{\gamma}(t),\dot{\widetilde{\gamma}}(t))\, dt = \inf_{\widehat{\gamma}}\int_{i/m}^{(i+1)/m} f(\widehat{\gamma}(t), \dot {\widehat{\gamma}}(t))dt,
								\label{segment_optimal}\end{eqnarray}
                where the infima are over Lipschitz paths $\widetilde{\gamma}:[0,1]\to D$ and $\widehat{\gamma}: [i/m, (i+1)/m] \to D$ with 
			$\widetilde{\gamma}(0)=v_i$, $\widetilde{\gamma}(1)=v_{i+1}$, $\widehat{\gamma}(i/m) = v_i$ and $\widehat{\gamma}((i+1)/m) = v_{i+1}$.
			
\medskip
{\it Relations between $G_n$ and $L_n$.}
At this point, we remark there are kernels $f$ for which $G_n = L_n$, namely those such that linear segments are in fact quasinormal geodesics.  An example is $f(x,v)=|v|$.  Identifying these kernels is a question with a long history, going back to Hilbert, whose 4th problem paraphrased asks for which geometries are the geodesics straight lines (cf. surveys \cite{Paiva}, \cite{papadopoulos}).  			
Hamel's criterion, namely $\partial_{x_i} \partial_{v_j} f = \partial_{x_j}\partial_{v_i} f$ for $1\leq i,j\leq d$, is a well-known solution to this question (see \cite{gelfand}, \cite{Paiva}, \cite{papadopoulos} and references therein).  

We also note, as mentioned in the introduction, that the case $d=p=1$ is `degenerate' in that $\min G_n$ and $\min L_n$ are not random.  Indeed, let $\gamma_{\v}\in \argmin G_n$ and suppose $\v= (v_0,\ldots, v_m)\in V_n(a,b)$.  We observe that $\gamma_\v$ must be nondecreasing, as otherwise, one could build a smaller cost path, from parts of $\gamma_\v$ using invariance to reparametrization, violating optimality of $\gamma_\v$.  In particular, $\dot\gamma_\v\geq 0$ and $v_i<v_{i+1}$ for $0\leq i\leq m-1$.  Then, 
$$ G_n(\gamma_\v) = \sum_{i=0}^{m-1} \int_{i/m}^{(i+1)/m} f(\gamma_\v(t),\dot\gamma_\v(t))dt
= \sum_{i=0}^{m-1} \int_{v_i}^{v_{i+1}} f(s, 1)ds = \int_a^b f(s,1)ds,$$
using the $1$-homogeneity of $f$ and changing variables.  The same argument yields that $\min L_n = \int_a^b f(s,1)ds$.  We do not consider this `degenerate' case further.
\medskip

The first result is for linearly interpolated paths.

\begin{theorem} \label{theorem1} 
Suppose that $p > 1$. With respect to
  realizations $\{ X_i\}$ in a probability $1$ set, the
  following holds. The minimum values of the costs $L_n$ converge
  to the minimum of $F$,
  \[ \lim_{n \to \infty} \min_{\gamma \in \Omega_n^l(a,b)} L_n(\gamma) = \min_{\gamma \in \Omega(a,b)} F(\gamma). \]

  Moreover, consider a sequence of optimal paths $\gamma_n \in \argmin L_n$.  Any subsequence of $\{\gamma_n\}$ has a further subsequence that converges
  uniformly to a limit path
  $\gamma \in \argmin F$,
  \[ \lim_{k \to \infty} \sup_{0 \leq t \leq 1} |\gamma_{n_k}(t) - \gamma(t)| = 0. \]
	
	In addition, if $\gamma$ is the unique minimizer of $F$, then the whole sequence $\gamma_n$ converges uniformly to $\gamma$.

\end{theorem}

The case $d\geq 2$ and $p=1$ requires further development, and is addressed with a few more assumptions in Theorem \ref{linear_p=1approxthm}.

\medskip

We now address quasinormal interpolations.

\begin{theorem} \label{theorem2} 
Suppose that either (1) $p> 1$ or (2) $d\geq 2$ and $p=1$. Then, with respect to 
  realizations $\{ X_i\}$ in a probability $1$ set, the
  following holds. The minimum values of the energies $G_n$ converge
  to the minimum of $F$,
  \[ \lim_{n \to \infty} \min_{\gamma \in \Omega_n^\gamma(a,b)} G_n(\gamma) = \min_{\gamma \in \Omega(a,b)} F(\gamma). \]

  Moreover, consider a sequence of optimal paths $\gamma_n \in \argmin G_n$.  Any subsequence of $\{\gamma_n\}$ has a further subsequence that converges
  uniformly to a limit path
  $\gamma \in \argmin F$,
  \[ \lim_{k\to \infty} \sup_{0 \leq t \leq 1} |\gamma_{n_k}(t) - \gamma(t)| = 0. \]

In addition, if $\gamma$ is the unique minimzer of $F$, then the whole sequence $\gamma_n$ converges uniformly to $\gamma$.
\end{theorem}

We remark, when $d\geq 2$ and $p=1$, that there is a certain ambiguity in the results of Theorem \ref{theorem2}, due to the invariance of $F$ under reparametrization of paths.  In this case, there is no unique minimizer of $F$.  Consider for example the case where $f(x,v)=|v|$ and $F(\gamma) = \int_0^1 |\dot{\gamma}(t)|dt$. Any minimizer of this functional is a parameterization of a line, but of course such minimizers are not unique.  

One way to address this is to formulate a certain Hausdorff convergence with respect to images of the paths.
Given $\gamma \in \Omega(a,b)$, we denote the image of $\gamma$ by
\begin{equation*}
  S_{\gamma} = \left\{ \gamma(t)\ \middle|\ 0 \leq t \leq 1 \right\}.
\end{equation*}
Consider the Hausdorff
metric $d_{haus}$, defined on compact subsets $A,B$ of $D$ by
\begin{equation*}
  d_{haus}(A,B) = \max\{ \sup_{x \in A} \inf_{y \in B} d(x,y), \sup_{y \in B} \inf_{x \in A} d(x,y) \}.
\end{equation*}

\begin{corollary} \label{hausdorffcor} 
Suppose that either (1) $d\geq 2$ and $p=1$ or (2) $p>1$.   
Consider paths $\{\gamma_{\v^{(n)}}\}$, for all large $n$ either in form $\gamma_{\v^{(n)}}\in \argmin G_n$, 
or $\gamma_{\v^{(n)}}\in \argmin L_n$.

 Then, with respect to realizations $\{X_i\}$ in a probability $1$ set, any subsequence of $\{\v^{(n)}\}$ has a further subsequence which converges in the Hausdorff sense to $S_{\gamma}$,
  where $\gamma \in \argmin F$ is an optimal path.
	
	Moreover, if $F$
  has a unique (up to reparametrization) minimizer $\gamma$, then the whole sequence converges,
  \begin{equation*}
    \lim_{n \to \infty} d_{haus}(\v^{(n)},S_{\gamma}) = 0.
  \end{equation*}

\end{corollary}

\subsection{Riemann sum costs and $p=1$-linear interpolating costs}
\label{Riemann_subsect}

\begin{figure}
 \includegraphics[trim={0.8cm 0.8cm 0.65cm 0.35cm},width=0.6\linewidth]{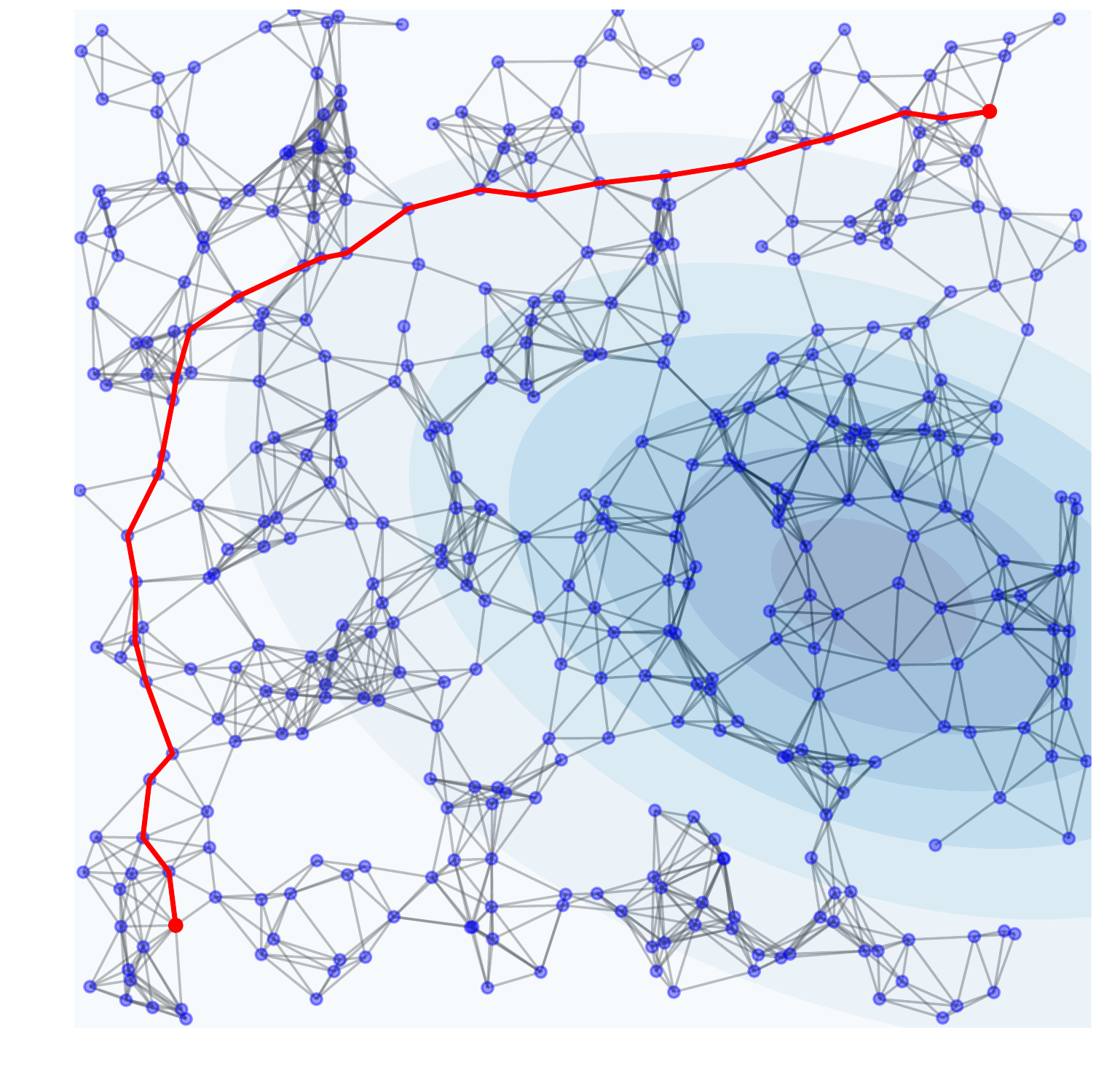}
 \label{fig:rgg_path}
 \caption{$H_{400}$-minimizing discrete path in the setting of Figure \ref{fig1}, linearly interpolated for visual clarity.}
 \label{fig2}
\end{figure}

We first introduce a cost which requires knowledge of $f$ only on discrete points and, as a consequence, more `applicable'.  At the end of the subsection, we return to linear interpolated costs when $p=1$.

  Define $H_n : V_n(a,b) \to \mathbb{R}$, for $\v = (v_0,v_1,\ldots, v_m)$, by
\begin{equation}
  H_n(\v) = \frac{1}{m} \sum_{i=1}^m f(v_i, m(v_{i+1}-v_i)).
\label{H_n_eq}
\end{equation}
The functional $H_n$ is, in a sense, a `Riemann sum' approximation to $L_n$ and $G_n$, and therefore its behavior, and the behavior of its minimizing paths, should be similar to that of $L_n$ and $G_n$.  See Figure \ref{fig2} for an example of an optimal $H_n$ path.

We make this intuition rigorous by establishing
variants of Theorems \ref{theorem1} and \ref{theorem2} with respect to
the cost $H_n$.   Given the `rougher' nature of $H_n$, however, additional assumptions on $f$ and $\epsilon_n$, beyond those in the `standing assumptions', will be helpful in this regard.
As in the previous Subsection, our results differ between the two cases $p = 1$ and
$p > 1$.

Define the following smoothness condition: 
\begin{itemize}
\item [(Lip)] There exists a $c$ such that for all
$x,y \in D$ and $v \in \mathbb{R}^d$ we have
\begin{equation*} \label{flipschitz}
      |f(x,v) - f(y,v)| \leq c |x-y| |v|^p.
\end{equation*}
\end{itemize}
We note when $f$ satisfies the homogeneity condition \eqref{fhomo}, and $\nabla_x f(x,v)$ is uniformly bounded on $D\times \{y: |y| = 1\}$, that (Lip) holds.

We now consider the behavior of $H_n$ when $p > 1$. 
The analogue to Theorem \ref{theorem1} and Corollary \ref{hausdorffcor} in this setting is the
following.

\begin{theorem}
\label{p>1approxthm}
	Suppose $p > 1$, and that $f$ in addition
  satisfies (Lip). With respect to realizations
  $\{ X_i\}$ in a probability 1 set, the minimum values of the energies $H_n$ converge to the
  minimum of $F$,
  \[ \lim_{n \to \infty} \min_{\v \in V_n(a,b)} H_n(\v) = \min_{\gamma \in \Omega(a,b)} F(\gamma). \]

  Further, consider a sequence of optimal discrete paths $\w^{(n)} \in \argmin H_n$, and their linear interpolations $\{l_{\w^{(n)}}\}$.  Then, for any subsequence of $\{l_{\w^{(n)}}\}$ and correspondingly of $\{\w^{(n)}\}$, there is a further subsequence of the linear paths which converges uniformly to a limit path
$\gamma \in \argmin F$, and of the discrete paths in the Hausdorff sense to $S_\gamma$.   

If $F$ has a unique minimizer $\gamma$, the whole sequence of linear paths converges uniformly to $\gamma$, and the whole sequence of discrete paths converges in the Hausdorff sense to $S_\gamma$.

\end{theorem}

We will need to impose further assumptions on the integrand $f$ to state results in the case $p=1$.  See below for examples of $f$ satisfying these conditions, and also Subsection \ref{remarks} for further comments. 

\begin{itemize}
\item [(Hilb)] We say that $f$ satisfies the `Hilbert condition' if, for each $x$, 
$$\inf_{\gamma \in \Omega(a,b)}\int_0^1 f(x, \dot\gamma(t))dt = f(x, b-a),$$ 
that is, straight lines are geodesics for the kernel $f(x,\cdot)$.

\item[(TrIneq)] We say $f$ satisfies the `triangle inequality' if, for each $x$, 
$$f(x,v-w)\leq f(x,v-u) + f(x,u-w)$$ 
for all $u,v,w \in \mathbb{R}^d$.

\item[(Pythag)] Let $\alpha>1$.  Consider points $u,v,w$ where $|uw|,|vw|,|uv|<\eta$ for an $\eta<1$.  Suppose there is a constant $c$ such that, for $0<r<1$,
\begin{itemize}
\item ${\rm dist}(w, {\rm line}(u,v))\geq r$, and 
\item $|uv|\leq cr^{1/\alpha}$. 
\end{itemize}
Then, we say $f$ satisfies the `Pythagoras $\alpha$-condition' if there is a constant $C=C(\alpha, f, c)$ such that, for all $x$,
$$f(x,w-u) + f(x,v-w)\geq f(x,v-u) + Cr^\alpha.$$  
Here, ${\rm line}(u,v)$ is the line segment between $u$ and $v$.

\end{itemize}

 Here, in the statement of (Hilb), the kernel function, for fixed $x$, is only a function of $v$.  The following lemma is a case of the Hamel's criterion discussed in the previous Subsection. 
\begin{lemma}
\label{Hilblemma}
Given the `standing assumptions', suppose also, for fixed $x\in D$, that $v\mapsto f(x,v)$ is $C^2$ on $\R^d\setminus\{0\}$ with positive definite Hessian.   Then, (Hilb) is satisfied.
\end{lemma}

\begin{proof} Fix an $x_0\in D$.  There is a quasinormal minimizer $\gamma\in C^2$ where both $c=f(x_0, \dot\gamma(t))=\inf_{\gamma\in \Omega(a,b)}\int_0^1f(x_0,\dot\gamma(t))dt$ and 
$$c^2=f^2(x_0,\dot\gamma(t)) = \inf_{\gamma\in\Omega(a,b)}\int_0^1 f^2(x_0, \dot\gamma(t))dt$$
 for a.e. $0\leq t\leq 1$ (cf. Prop. 5.25 in \cite{buttazzo1998one}).  Let $g(v) = f^2(x_0,v)$. Then, $\gamma$ satisfies the Euler-Lagrange equation $\frac{d}{dt}\nabla_v g( \dot\gamma(t)) = \nabla_x g(\dot\gamma(t)) = 0$.  In other words, $H\ddot\gamma(t) = 0$, where $H$ denotes the Hessian of $g$. By assumption, $H$ is positive definite.  Hence, $\ddot\gamma(t)\equiv 0$, and so $\gamma$ is a parametrization of a straight line.
\end{proof}

An example of a class of kernels $f$ satisfying the `standing assumptions' and the additional conditions above is given in the following result.
Recall $\langle\cdot, \cdot\rangle$ denotes the Euclidean inner product on $\R^d$.

\begin{lemma}\label{h_condition}
Let $x\mapsto M=M(x)$ be a $C^1$, strictly elliptic matrix-valued function on $D$.
The kernel $f(x,v) = 
\langle v, M(x)v\rangle^{1/2}$
satisfies the `standing assumptions', and also (Lip), (Hilb), (TrIneq) and (Pythag) for all $\alpha>1$.  
\end{lemma}

\begin{proof}
The kernel clearly satisfies the `standing assumptions' and (Lip).
Next, for fixed $x$, the map $v\mapsto f(x,v)$ satisfies the conditions of Lemma \ref{Hilblemma}, and so satisfies (Hilb).  Also, $v\mapsto f(x,v)$ trivially satisfies (TrIneq).

We show (Pythag) in the case $f(x,v) = h(x)|v|$, that is $M(x) = h^2(x)Id$, as the notation is easier and all the ideas carry over to the more general case.  Consider a right triangle joining $u,w, z$ where $z$ is on the line through $u,v$ (cf. Figure \ref{fig:pythag}).
\begin{figure}[h]
\centering
\includegraphics[trim={0cm 1cm 0cm 1cm}]{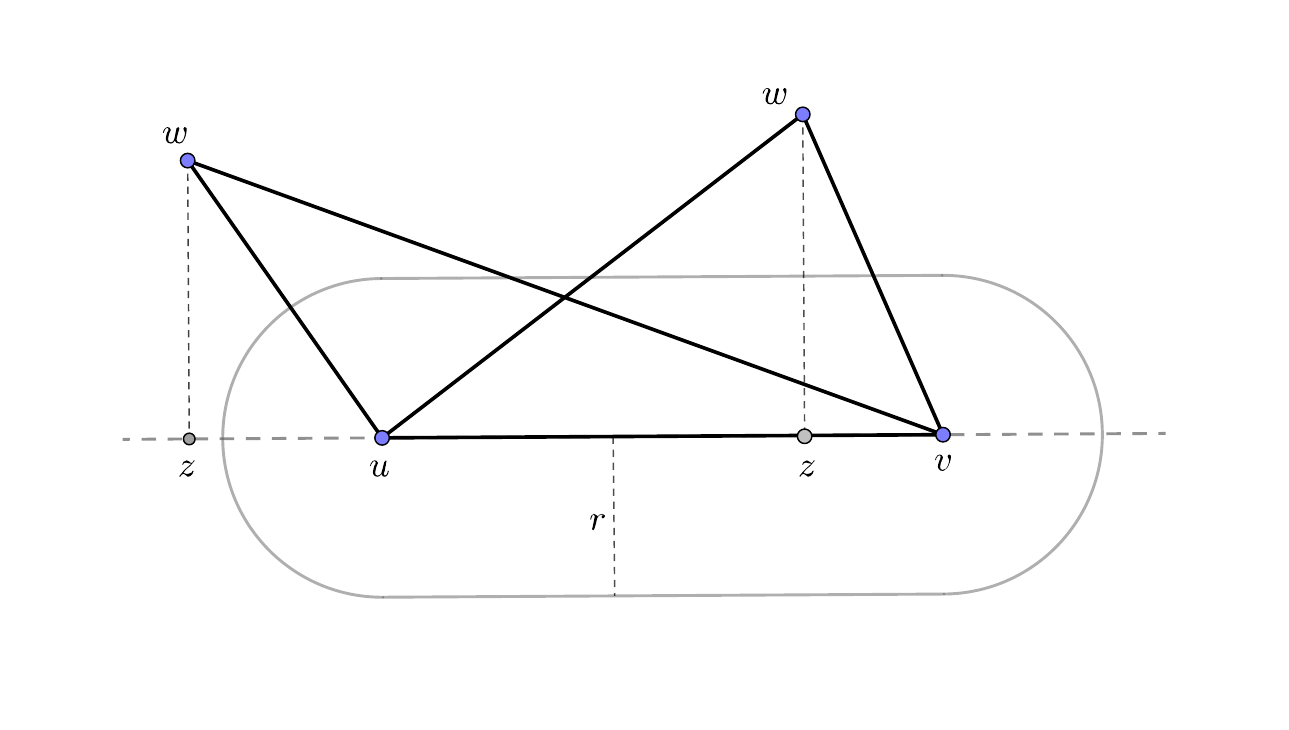}
\caption{Geometric argument used in proof of Lemma \ref{h_condition} with respect to (Pythag).}
\label{fig:pythag}
\end{figure}

If $z$ is not on the line segment connecting $u$ and $v$, then either (a) $|uw|\geq |uv|$ and $|wv|\geq r\geq r^\alpha$ or (b) $|wv|\geq |uv|$ and $|uw|\geq r\geq r^\alpha$.  In either case, $h(x)[|uw|+|wv|] \geq h(x)|uv| + m_1r^\alpha$ and (Pythag) is satisfied.

Suppose now $z$ is on the line segment connecting $u$ and $v$.  In the triangle,  $uw$ is the hypotenuse, and $wz$ and $uz$ are the legs, such that $|uw|^2 = |wz|^2 + |uz|^2$.  Hence, as $|uw|<1$, all the lengths are less than $1$.  We are given that $|wz|\geq r$ and $|uz|\leq |uv|\leq cr^{1/\alpha}$.  Then, as $r\leq |wz|<1$ and $2<\min\{2\alpha, \alpha +1/\alpha \}$, we have
\begin{eqnarray*}
|uw|^2 &=& |uz|^2 + |wz|^2\\
& \geq & |uz|^2 + (1/9)r^{2\alpha} + 2(1/3)r^{\alpha+ 1/\alpha}\\
&\geq& |uz|^2 + (9\max\{c,1\}^2)^{-1}r^{2\alpha} + 2(3\max\{c,1\})^{-1}|uz|r^\alpha\\
&=&  \big(|uz| + (3\max\{c,1\})^{-1}r^\alpha\big)^2,
\end{eqnarray*}
and so $|uw| \geq |uz| + (3\max\{c,1\})^{-1}r^\alpha$.

A similar inequality, $|wv| \geq |vz| + (3\max\{c,1\})^{-1}r^\alpha$, holds with the same argument.  Hence,
$h(x)|uw| + h(x)|wv| \geq h(x)|uv| + C(m_1,c)r^\alpha$.
\end{proof}

We will also need to limit the decay properties of $\epsilon_n$ for the next result; see Subsection \ref{remarks} for comments on this limitation.
Namely, we will suppose that $\epsilon_n$ is in form $\epsilon_n = n^{-\delta}$ where 
\begin{equation}
\label{additionaleps}
\delta>\max\{[(2-\alpha^2)\eta + d]^{-1}, [(\alpha(d-1) + 1]^{-1}\},
\end{equation}
 for an $0<\eta<1$ and $1<\alpha<\sqrt{2}$.

We note that condition \eqref{rateassumption}, when $\epsilon_n$ is in form $\epsilon_n = n^{-\delta}$, yields that $\delta< 1/d$.  However, when $d\geq 2$, we have $\max\{[(2-\alpha^2)\eta + d]^{-1}, [(\alpha(d-1) + 1]^{-1}\} < 1/d$, and so \eqref{additionaleps}, in conjuction with \eqref{rateassumption}, limits $\delta$ to an interval.

\begin{theorem}
\label{p=1approxthm}
	Suppose $d\geq 2$ and  $p=1$, and that $f$ also satisfies (Lip) and (Hilb).
	With respect to realizations $\{ X_n\}_{n\geq 1}$ in a
  probability $1$ set, the minimum values of the
  cost $H_n$  
  converge to the minimum of $F$,
  \[ \lim_{n \to \infty} \min_{\v \in V_n(a,b)} H_n(\v) = 
  \min_{\gamma \in \Omega(a,b)} F(\gamma). \]

Moreover, suppose now that $f$ in addition satisfies (TrIneq) and (Pythag) for an $1<\alpha<\sqrt{2}$, and that $\epsilon_n$ satisfies \eqref{additionaleps}.

Consider a sequence of optimal discrete paths $\v^{(n)} \in \argmin H_n$, and their linear interpolations $\{l_{\v^{(n)}}\}$.  Then, for any subsequence of $\{l_{\v^{(n)}}\}$ and so
of $\{\v^{(n)}\}$, there is a further subsequence of the linear paths which converges uniformly to a limit path
  $\gamma \in \argmin F$, and of the discrete paths in the Hausdorff sense to $S_\gamma$.   

If $F$ has a unique (up to reparametrization) minimizer $\gamma$, the whole sequence of discrete paths converges, $\lim_{n\rightarrow\infty}d_{haus}(\v^{(n)}, S_\gamma) = 0$.
\end{theorem}

See Figure \ref{fig2} for an example of an $H_n$-cost geodesic path.

As noted in the introduction, when $d=p=1$, $H_n(\v)$ is a certain Riemann sum.  Let $\w\in \argmin H_n$, and observe by optimality that $\w= (w_0,\ldots, w_m)\in V_n(a,b)$ satisfies $w_i<w_{i+1}$ for $0\leq i\leq m-1$.  Hence, by $1$-homogenity of $f$, $H_n(\w) = \sum_{i=0}^{m-1} f(w_i, 1)|w_{i+1}-w_i|$, and $H_n(\w)$ strongly approximates the integral $\int_a^b f(s, 1)ds$, given that the partition length $\max |w_{i+1}-w_i|\leq \epsilon_n \to 0$.  For this reason, this case is not included in the above theorem.

\medskip
\noindent
{\it Linear interpolating costs when $p=1$.}  Having now introduced (Lip), (Hilb), (TrIneq) and (Pythag), we address the case $d\geq 2$ and $p=1$ with respect to the cost $L_n$.

\begin{theorem}
\label{linear_p=1approxthm}
	Suppose $d\geq 2$ and  $p=1$, and that $f$ also satisfies (Lip) and (Hilb).
	With respect to realizations $\{ X_n\}_{n\geq 1}$ in a
  probability $1$ set, the minimum values of the
  cost $L_n$  
  converge to the minimum of $F$,
  \[ \lim_{n \to \infty} \min_{\gamma \in \Omega^l_n(a,b)} L_n(\gamma) = 
  \min_{\gamma \in \Omega(a,b)} F(\gamma). \]

Moreover, suppose now that $f$ in addition satisfies (TrIneq) and (Pythag) for an $1<\alpha<\sqrt{2}$, and that $\epsilon_n$ satisfies \eqref{additionaleps}.

Consider a sequence of optimal paths $l_{\v^{(n)}} \in \argmin L_n$.
 Then, for any subsequence of $\{l_{\v^{(n)}}\}$ and so 
 of the discrete paths $\{\v^{(n)}\}$, there is a further subsequence of the linear paths which converges uniformly to a limit path
  $\gamma \in \argmin F$, and of the discrete paths in the Hausdorff sense to $S_\gamma$.   

If $F$ has a unique (up to reparametrization) minimizer $\gamma$, the whole sequence of discrete paths converges, $\lim_{n\rightarrow\infty}d_{haus}(\v^{(n)}, S_\gamma) = 0$.
\end{theorem}

\subsection{Remarks}  We make several comments about the assumptions and related issues.
\label{remarks}

\vskip .1cm {\textbf 1.} {\it Domain.}  The requirements that $D$
should be closed and connected are needed for the quasinormal path
results in \cite{buttazzo1998one} and \cite{hildebrandt} to hold.
Also, the proof of Proposition \ref{nearestneighbor}, on the maximum
distance to a nearest neighbor vertex, requires that the domain boundary
should be Lipschitz, true for convex domains. The convexity of the domain also ensures that all the linearly interpolated paths are within the domain, and allows comparison with quasinormal ones, which by definition are constrained in the domain, as in the proof of the `limsup' inequality, Lemma \ref{limsup}.   In addition, a bound on
the domain allows the Arzela-Ascoli equicontinuity criterion to be
applied in the compactness property, Lemma \ref{compactness}.

\vskip .1cm

{\textbf 2.} {\it Ellipticity of $\rho$.}  The bound on $\rho$ is useful to compare $\nu$ to the uniform distribution in the nearest-neighbor map result, Proposition \ref{nearestneighbor}, as well as in bounding the number of points in certain sets in Lemma \ref{kbound}.  We note, as our approximating costs, $L_n$, $G_n$, $H_n$, do not involve density estimators, our results do not depend on the specifics of $\rho$, unlike for `density based distances' discussed in \cite{Sajama}.

\vskip .1cm

{\textbf 3.}  {\it Decay of $\epsilon_n$ \eqref{rateassumption}.}  Intuitively, the rate $\epsilon_n$ cannot vanish too quickly, as then the graph may be disconnected with respect to a postive set of realizations $\{X_i\}$.  However, the estimate in \eqref{rateassumption} ensures that the graph $\mathcal{G}_n(a,b)$ is connected for all large $n$ almost surely--see Proposition \ref{nearestneighbor}.  This is a version of the `$\delta$-sampling' condition in \cite{bernstein2000graph}, and is related to connectivity estimates in continuum percolation \cite{penrose}.  Moreover, we note, the prescribed rate yields in fact that any vertex $X_i$ will have degree tending to infinity as $n$ grows, as long as $\rho$ is elliptic:  One calculates that the mean number of points in the $\epsilon_n$ ball around $X_i$ is of order $n\epsilon_n^d$ which grows faster than $\log(n)$.

\vskip .1cm

{\textbf 4.}  {\it Assumptions (A0)-(A3) on $f$.}  These are somewhat standard assumptions to treat parametric variational integrals such as $F$ (cf. \cite{buttazzo1998one} and \cite{hildebrandt}), which include the basic case $f(x,v) = |v|^p$.

\vskip .1cm

{\textbf 5.}  {\it Assumption on $p$.}  The assumption $p\geq 1$ is useful to show existence of quasinormal paths in Proposition \ref{existence}, and compactness of minimizers.   The case $p<1$ is more problematic in this sense and not discussed here.  

\vskip .1cm

 {\textbf 6.} {\it Extra assumptions in Theorems \ref{p=1approxthm} and \ref{linear_p=1approxthm}. } The main difficulty is in showing compactness of optimal $H_n$ and $L_n$ paths when $p=1$.  With respect to Theorem \ref{p>1approxthm}, when $p>1$, the form of the cost allows a Holder's inequality argument to deduce equicontinuity of the paths, from which compactness follows using Ascoli-Arzela's theorem.  However, there is no such coercivity when $p=1$.  Yet, with the additional assumptions, one can approximate a geodesic locally by straight lines.  Several geometric estimates on the number of points in small windows around these straight lines are needed to ensure accuracy of the approximation, for which the upperbound on $\epsilon_n$ in \eqref{additionaleps} is useful.

\vskip .1cm

{\textbf 7.} {\it Unique minimizers of $F$.}
Given that our results achieve their strongest form when $\argmin F$ consists of a unique minimizing path, perhaps up to reparametrization, we comment on this possibility.  Under suitable smoothness conditions on the integrand $f$, uniqueness criteria for ordinary differential equations allow to deduce from the Euler-Lagrange equations, $d/dt \nabla_v f(\gamma(t), \dot\gamma(t)) = \nabla_x f(\gamma(t), \dot\gamma(t))$, that there is a unique geodesic between points $a,b$ sufficiently close together (cf. Proposition 5.25 in \cite{buttazzo1998one} and Chapter 5 in \cite{Burago}).  On the other hand, for general $a,b$, `nonuniqueness' may hold depending on the structure of $f$.  For instance, one may construct an $f$, satisfying the `standing assumptions', with several $F$-minimizing paths, by penalizing portions of $D$ so as to induce `forks'.

\vskip .1cm

{\textbf 8.} {\it $k$-nearest neighbor graphs.} It is not clear if our approximation results, say Theorem \ref{p=1approxthm}, hold with respect to the $k$-nearest-neighbor graph with $k$ bounded--that is the graph formed by attaching edges from a vertex to the nearest $k$ points.  For instance, when $\{X_i\}$ is arranged along a fine regular grid, $f(x,v)\equiv |v|$, $d=2$, and $k=4$, the optimal $H_n$ route of moving from the origin to $(1,1)$ is on a `staircase' path with length $\sim{2}$, no matter how refined the grid is, yet the Euclidean distance is $\sqrt{2}$.  In this respect, the random geometric graph setting of Theorem \ref{p=1approxthm} allows enough choices among nearby points, as long as $\rho$ is elliptic, for the optimal path to approximate the straight line from $(0,0)$ to $(1,1)$.  It would be of interest to investigate the extent to which our results extend to $k$-nearest neighbor graphs.

\vskip .1cm

\section{Proof of Theorems \ref{theorem1}, \ref{theorem2}, and Corollary \ref{hausdorffcor}}
\label{interpolating_proofs}

As mentioned in the introduction, the proof of Theorems \ref{theorem1} and \ref{theorem2} relies on a probabilistic `Gamma
Convergence' argument. After establishing some basic notation and results on quasinormal minimizers, we
present three main proof elements, `liminf inequality', `limsup inequality' and `compactness', in the following Subsections.
Proofs of Theorems \ref{theorem1} and \ref{theorem2}, and Corollary \ref{hausdorffcor} are the end of the Subsection.
\subsection{Preliminaries}

Define a `nearest-neighbor' map $T_n: D \to \mathcal{X}_n$ where, for $x\in D$,
$T_n(x)$ is the point of $\mathcal{X}_n$ closest to $x$ with
respect to the Euclidean distance. In the event of a tie, we adopt the
convention that $T_n(x)$ is that nearest neighbor in $\mathcal{X}_n$
with the smallest subscript.  Since $\mathcal{X}_n$ is random, we note $T_n$ and the distortion
$$\|T_n - Id\|_{\infty} = \sup_{x \in D} |T_n(x) - x|
=\|T_n - Id\|_{\infty} = \sup_{y \in D} \min_{1 \leq i \leq n} |X_i -
y|$$
are also random. In Proposition \ref{nearestneighbor} of the appendix, we show for $a,b\in D$ that, almost surely, 
\begin{equation}
\label{graph_line}
\text{the graph }\mathcal{G}_n(a,b) \text{ is connected for all } a,b\in D \text{ and all large }n.
\end{equation}
  Moreover, it is shown there that
exists a constant $C$ such that almost surely,
\begin{equation} \label{tnrate}
 \limsup_{n \to \infty} \frac{\|T_n - Id\|_{\infty} n^{1/d}}{(\log n)^{1/d}} \leq C.
\end{equation}
Throughout, we will be working with realizations where both \eqref{graph_line} and \eqref{tnrate} are satisfied.  Let 
\begin{equation*}
A_1 \text{ be the probability }1\text{ event that } \eqref{graph_line} \text{ and } \eqref{tnrate} \text{ hold}.
\end{equation*}
We observe, when the decay rate \eqref{rateassumption} on $\epsilon_n$ holds, on the set of realizations $A_1$, we have
$\lim_{n\rightarrow\infty} \|T_n - Id\|_\infty/\epsilon_n = 0$.

To rule out certain degenerate configurations of points, in $d\geq 2$, let
\begin{equation*}
\label{A_2}
  A_2 \text{ be the event that } X_k \not \in S_{\gamma_{X_i,X_j}} \forall \text{ distinct } i,j,k \in \mathbb{N}.
\end{equation*}
Since the $\{X_i\}$ come from a continuous distribution, and the image of the Lipschitz path $\gamma_{X_i,X_j}$ in $D\subset \mathbb{R}^d$, when $d\geq 2$, is of lower dimension, $A_2$ has
probability $1$.

Recall the definitions of quasinormal and linear paths $\gamma_{u,v}$ and $l_{u,v}$.

\begin{prop} \label{lgammacompare} Let $m_1$ and $m_2$ be the
  constants in \eqref{fbounds}. Then for $u,v\in D$,
    \begin{equation} \label{dfcomparison}
    m_1 | u - v|^p  \leq d_f(u,v) \leq m_2 |u - v|^p.
  \end{equation}
  Further, the path $\gamma_{u,v}$ satisfies, for $0\leq s,t\leq 1$, that
  \begin{equation} \label{uniformlipschitz} |\gamma_{u,v}(s) -
    \gamma_{u,v}(t)| \leq \big({m_2}/{m_1})^{1/p} |u-v| |s-t|,
  \end{equation}
  and
  \begin{equation} \label{segmentdistance}
    \sup_{0 \leq t \leq 1} |\gamma_{u,v}(t) - l_{u,v}(t)| \leq \Big(\big({m_2}/{m_1}\big)^{1/p} + 1\Big) |u-v|.
  \end{equation}
\end{prop}

\begin{proof}

For a Lipschitz path $\gamma$ between $u$ and $v$, we have $ m_1 |\dot\gamma(t)|^p \leq f(\gamma(t),\dot\gamma(t))\leq m_2|\dot\gamma(t)|^p$ by \eqref{fbounds}.  Also, $\inf_{\gamma\in \Omega(a,b)} \int_0^1|\dot\gamma(t)|^pdt = |b-a|^p$ by a standard calculus of variations argument (see also Proposition \ref{existence}).  So, by taking infimum over $\gamma$, we obtain \eqref{dfcomparison}.

  Suppose now $\gamma=\gamma_{u,v}$ is quasinormal, so that
  $f(\gamma(t),\dot{\gamma}(t)) = c$ for some constant $c$
  and a.e. $t$. Integrating, and noting \eqref{dfcomparison}, gives $c = \int_0^1 f(\gamma,\dot{\gamma})\, dt = d_f(u,v)\leq m_2|u-v|^p$.
	On the other hand, by \eqref{fbounds}, $m_1|\dot\gamma(t)|^p\leq f(\gamma(t),\dot\gamma(t)) = c$.  Hence, $|\dot\gamma(t)|\leq (m_2/m_1)^{1/p}|u-v|$ and \eqref{uniformlipschitz} follows.

  Finally, to establish \eqref{segmentdistance}, suppose that there is
  a $t$ such that
  $|\gamma_{u,v}(t) - l_{u,v}(t)| > \big((m_2/m_1)^{1/p} + 1\big)|u-v|$. Then,
  considering that $|l_{u,v}(t) - v| \leq |u-v|$, an application of
  the triangle inequality gives
  $|\gamma_{u,v}(t) - v|  > (m_2/m_1)^{1/p}|u-v|$.   However, by \eqref{uniformlipschitz}, 
  $|\gamma_{u,v}(t) -v| = |\gamma_{u,v}(t) - \gamma_{u,v}(1)| \leq (m_2/m_1)^{1/p} |u-v||t-1|\leq (m_2/m_1)^{1/p}|u-v|$, a contradiction. 
Thus, inequality \eqref{segmentdistance} holds.
\end{proof}

\subsection{Liminf Inequality}

A first step in getting some control over the limit cost $F$ in
terms of the discrete costs is the following bound.

\begin{lemma}[Liminf Inequality] \label{liminf} Consider $\gamma \in \Omega(a,b)$, and suppose
  we have a sequence of paths $\gamma_n \in \Omega(a,b)$ such that
  $$
    \lim_{n \to \infty} \sup_{0 \leq t \leq 1} |\gamma_n(t) - \gamma(t)| = 0 \hspace{0.5cm} \text{ and } \hspace{0.5cm}
    \sup_n \int_0^1 |\dot{\gamma}_n|^p\, dt < \infty.
  $$
  Then,
  $F(\gamma) \leq \liminf_{n \to \infty} F(\gamma_n)$.
\end{lemma}

\begin{proof}
  A sufficient condition for this inequality, a `lower
  semicontinuity' property of $F$, to hold is that $f(x,v)$ be
  jointly continuous and convex in $v$. See Theorem 3.5 (and the
  subsequent Remark 2) of \cite{buttazzo1998one} for more discussion
  on this matter.
\end{proof}

\subsection{Limsup Inequality}

To make effective use of the liminf inequality, we need to identify a
sufficiently rich set of sequences for which a reverse inequality holds. 
To this end, we develop certain approximations of Lipschitz paths by
piecewise linear or piecewise quasinormal paths.

The following result gives a method for recovering an element of
$V_n(a,b)$ from a suitable element of $\Omega(a,b)$.

\begin{prop} \label{ctwoapprox}
Suppose, for constants $c,C$, that $\gamma \in \Omega(a,b)$ satisfies
\begin{equation} \label{growthbound}
  c \leq \frac{|\gamma(s) - \gamma(t) |}{|s-t|} \leq C, 
\end{equation}
for all $0 \leq s < t \leq 1$.  Let $N=N(n) = \ceil{K/\epsilon_n}$,
where $K=C+1$ say, and define
$v_0 = a, v_N = b$, and $v_i = T_n \gamma(i/N)$ for $0 < i < N$. 

Then,
with respect to realizations $\{X_i\}$ in the probability 1 set $A_1$, we have $\v=(v_0,\ldots,v_N) \in V_n(a,b)$ for all sufficiently large $n$.
\end{prop}

\begin{proof}
To show that $\v \in V_n(a,b)$, it is
  sufficient to verify that consecutive vertices $v_{i-1}$ and $v_{i}$ are connected by an edge
  in $\mathcal{G}_n(a,b)$, or in other words
  \begin{equation*}
    0 < |v_i - v_{i-1}| < \epsilon_n, \hspace{0.5cm} \text{for } i = 1,\ldots,N.
  \end{equation*}

  We first show that $|v_i - v_{i-1}| < \epsilon_n$.  Note that
  $|\gamma(i/N) - \gamma((i-1)/N)| \leq C/N \leq (C/K)
  \epsilon_n$ and $C/K<1$. For $1 < i < N$, we have
  \begin{align*}
    |v_i - v_{i-1}| &= |T_n \gamma(i/N) - T_n \gamma((i-1)/N)| \\
    &\leq (C/K)\epsilon_n +  2 \|T_n - Id\|_{\infty}.
  \end{align*}
  Similarly, for segments incident to an endpoint $a$ or $b$, we have
  \begin{align*}
    \max(|v_1 - v_0|, |v_N - v_{N-1}|) \leq (C/K)\epsilon_n + \|T_n - Id\|_{\infty}.
  \end{align*}
  In either case, assumption \eqref{rateassumption} on the decay of $\epsilon_n$ implies that, for realizations $\{X_i\}$ in $A_1$,  we
  have $|v_i - v_{i-1}| < \epsilon_n$ for all $1\leq i\leq N$ and
  sufficiently large $n$. 
  
  Now, we show that $0 < |v_i - v_{i-1}|$. By the Lipschitz lower bound on $\gamma$, we have
  \begin{align*}
    |\gamma(i/N) - \gamma((i-1)/N)| \geq c/N > (c/(K+1))\epsilon_n
  \end{align*}
  for $1\leq i\leq N$.
  By a triangle inequality argument, the distance between $v_i$ and $v_{i-1}$ is bounded below
  by $(c/(K+1))\epsilon_n - 2\|T_n - Id\|_{\infty}$, which on the set $A_1$, as $\epsilon_n$ satisfies \eqref{rateassumption} and therefore vanishes slower than $\|T_n - Id\|_{\infty}$, is positive for all large $n$.
\end{proof}

 We now establish
some approximation properties obtained by interpolating paths between points in $\v=(v_0,\ldots, v_N)$.

\begin{prop} \label{uniformrecovery} Fix $\gamma \in \Omega(a,b)$
  satisfying \eqref{growthbound}, and a realization $\{X_i\}$ in the probability $1$ set $A_1$. Let $\gamma_n = \gamma_{\v}$ and
  $l_n = l_{\v}$, where $N=N(n)$ and $\v=(v_0,\ldots,v_N)$ are defined as
  in Proposition \ref{ctwoapprox}. Then, we obtain
  \begin{align*}
    \lim_{n \to \infty} \sup_{0 \leq t \leq 1} |\gamma_n(t) - \gamma(t)| = 0,
    \text{   and   }
    \lim_{n \to \infty} \sup_{0 \leq t \leq 1} |l_n(t) - \gamma(t)| = 0.
  \end{align*}
  In addition,
  \begin{equation} \label{recoverybounds}
    \sup_{n} \|l_n'\|_{\infty}  <\infty, \text{ and }\ \ l_n'(t)\to \gamma'(t)  \text{ for a.e. } t \in [0,1].
\end{equation}

\end{prop}
\begin{proof}
  We first argue that
  $\lim_{n \to \infty} \sup_{0 \leq t \leq 1} |l_n(t) - \gamma(t)| = 0$. Let
  $u_i = \gamma(i/N)$, and let $\tilde{l}_n =l_{\u^{(n)}}\in \Omega(a,b)$ be the
  piecewise linear interpolation of $\u^{(n)}=(u_0,\ldots,u_N)$. As $\gamma$ is Lipschitz and $\lim_{n \to \infty} N(n) = \infty$, we have
  $\lim_{n \to \infty} \sup_{0 \leq t \leq 1} |\tilde{l}_n(t) - \gamma(t)| = 0$, and also $\lim_{n \to \infty} \tilde{l}_n'(t) = \gamma'(t)$ for a.e. $t \in [0,1]$. By
  construction, $l_n(i/N) = v_i= T_n\gamma(i/N)$ and $\tilde{l}_n(i/n) = u_i = \gamma(i/N)$ so that
  \begin{equation*}
    \max_{0\leq i\leq N} |l_n(i/N) - \tilde{l}_n(i/N)| \leq \|Id - T_n\|_{\infty}.
  \end{equation*}
  Then, as $l_n$ and $\tilde{l}_n$ are piecewise linear,
  it follows that
  $\sup_{0 \leq t \leq 1} |l_n(t) - \tilde{l}_n(t)| \leq \|Id - T_n
  \|_{\infty}$ and, as $\|Id - T_n \|_{\infty}$ vanishes on $A_1$, that
  $\lim_{n \to \infty} \sup_{0 \leq t \leq 1} |l_n(t) - \gamma(t)| = 0$.

  For $i/N < t < (i+1)/N$, we have \[l_n'(t) = N(v_{i+1} - v_i). \]
  As $|v_{i+1} - v_i| \leq \epsilon_n$ (Proposition \ref{ctwoapprox}), it follows that
  $|l_n'(t)| \leq N \epsilon_n \leq K+\epsilon_n$.  Hence,
  $\sup_n \|l_n'\|_{\infty} < \infty$.

 Likewise,
  $\tilde{l}_n'(t) = N(u_{i+1} - u_i)$, and so
  $$
    |l_n'(t) - \tilde{l}_n'(t)| \leq N(|v_{i+1} - u_{i+1}| + |v_i - u_i|) \leq 2N \|T_n - Id\|_{\infty}.
  $$
  For realizations in the probability $1$ set $A_1$, since $N = \ceil{K/\epsilon_n}$ and $\epsilon_n$ satisfies \eqref{rateassumption} and therefore vanishes slower than $\|T_n - Id\|_\infty$, we have $\lim_{n \to \infty} N\|T_n - Id\|_{\infty} = 0$.  Hence, $l_n'(t) \to \gamma'(t)$ for a.e. $t \in [0,1]$.

  Now, considering  the bound \eqref{segmentdistance}, it follows that
  \begin{equation*} \label{segmentsdistance}
    \sup_{0 \leq t \leq 1} |\gamma_n(t) - l_n(t) | \leq \max\left(C|v_0-v_1|,\ldots,C|v_{N-1} - v_N|\right) \leq C\epsilon_n,
  \end{equation*}
and hence $\|\gamma_n - \gamma\|_\infty \rightarrow 0$.
\end{proof}

With the above work in place, we proceed to the main result of this subsection.

\begin{lemma}[Limsup Inequality] \label{limsup} Let
  $\gamma \in \Omega(a,b)$ satisfy inequality
  \eqref{growthbound}. Then, with respect to realizations $\{X_i\}$ in the probability $1$ set
  $A_1$, we may find a sequence of paths $\{\gamma_n\}$ taken either in form for all large $n$ as (1) $\gamma_n \in \Omega_n^l(a,b)$ or (2) $\gamma_n\in \Omega_n^\gamma(a,b)$ such that
  $\lim_{n \to \infty} \sup_{0 \leq t \leq 1} |\gamma_n(t) -
  \gamma(t)| = 0$ and
  \begin{equation}
  \label{limsupineq}
    F(\gamma) \geq \limsup_{n \to \infty} F(\gamma_n).
   \end{equation}

\end{lemma}

We remark that the sequence $\{\gamma_n\}$ in the last lemma is called the `recovery sequence' since the liminf inequality in Lemma \ref{liminf} and the limsup inequality in Lemma \ref{limsup} together imply the limit, $\lim_n F(\gamma_n)=F(\gamma)$.

\begin{proof}
   Let
  $N = \ceil{K/\epsilon_n}$, where $K=C+1$ say is a constant greater than
  $C$ in \eqref{growthbound}.
	Define $v_0 = a, v_N = b$, and
  $v_i = T_n \gamma(i/N)$ for $0 < i < N$. Then, by Proposition
  \ref{ctwoapprox}, $\v=\v^{(n)}=(v_0,\ldots,v_N) \in V_n(a,b)$.

  We now consider paths in case (1). By Proposition \ref{uniformrecovery}, the
  interpolated paths $l_n = l_{\v} \in \Omega_n^l(a,b)$
  converge uniformly to $\gamma$. Consider the bound
  \begin{equation*}
    |F(l_n) - F(\gamma)| \leq \int_0^1 |f(l_n(t),l_n'(t)) - f(\gamma(t),\gamma'(t))|\, dt.
  \end{equation*}
    By Proposition \ref{uniformrecovery}, $l_n'$ converges
  almost everywhere to $\gamma'$, and
  $\sup_n \|l_n'\|_{\infty} < \infty$.  Hence $(l_n(t),l_n'(t)) \to (\gamma(t),\gamma'(t))$ for almost every
  $t$.  Also, $\|\dot\gamma\|_\infty< C$ by \eqref{growthbound}. Since, by \eqref{fbounds}, $f(x,v) \leq m_2 |v|^p$, an
  application of the bounded convergence theorem yields
  $\lim_{n \to \infty} |F(l_n) - F(\gamma)| = 0$.
Here, $\{l_n\}$ is the desired `recovery' sequence.

  We now consider case (2). Let
  $\gamma_n = \gamma_{\v}\in \Omega^\gamma(a,b)$. By Proposition
  \ref{uniformrecovery}, it follows that
  $\lim_{n \to \infty} \sup_{0 \leq t \leq 1} |\gamma_n(t) -
  \gamma(t)| = 0$. To show \eqref{limsupineq}
  for this sequence, write 
	\begin{align*}
   F(\gamma_n) &= \int_0^1 f(\gamma_n(t),\dot{\gamma_n}(t))\, dt \\
             &= \sum_{i=1}^N \int_{(i-1)/N}^{i/N} f(\gamma_n(t),\dot{\gamma_n}(t))\, dt \\
              &\leq \sum_{i=1}^N \int_{(i-1)/N}^{i/N} f(l_n(t),\dot{l_n}(t))\, dt \ = \ F(l_n),
\end{align*}
as $\gamma_n$ on the time interval $[i/N, (i+1)/N]$ corresponds to the minimum cost, geodesic path moving from $v_i$ to $v_{i+1}$ (cf. \eqref{segment_optimal}), and $l_n$ is a possibly more expensive path.

But, by case (1), $\limsup F(\gamma_n) \leq \limsup F(l_n) \leq F(\gamma)$.
\end{proof}

  \subsection{Compactness}

  In this Subsection, we consider circumstances under which a
  sequence of paths $\{\gamma_n\}$, in the context of Theorems \ref{theorem1} and \ref{theorem2}, has a limit
  point with respect to uniform convergence. Here, the arguments when $p = 1$ differ from those when $p > 1$.

  In particular, consider paths $\gamma_n$ where
  $\int_0^1 f(\gamma_n(t),\dot{\gamma_n}(t))\, dt$ is uniformly
  bounded. One has
  $m_1 |v|^p \leq f(x,v)$ and it follows that
  $\{\gamma_n \}$ is bounded in the $W^{1,p}$ Sobolev
  space. When $p > 1$, this is sufficient to derive a suitable
  compactness result. But, when $p=1$, this is no longer the case.

	However, when $p=1$, our general
  outlook is that it is enough to establish a compactness result for
  {\it sequences of optimal paths}, on which certain eccentric possibilities are ruled out.
	
	  We begin by considering such compactness when $p = 1$, when the paths lie in $\Omega_n^\gamma(a,b)$.  The setting $p>1$ is discussed afterwards.
  
  \begin{prop}  \label{p=1compactnesslemma}
		Suppose $d\geq 2$ and that $p=1$.  Then, with respect to realizations $\{X_i\}$ in the probability $1$ set $A_1\cap A_2$,
		for all large $n$, if $\gamma \in \argmin G_n$ and $0\leq s,t\leq 1$, we have that
    \begin{equation*} \label{lipschitz}
      |\gamma(s) - \gamma(t)| \leq (4m_2/m_1^2) G_n(\gamma) |s-t|.
    \end{equation*}
    
  \end{prop}

  \begin{proof}
    The path $\gamma \in \Omega_n^{\gamma}(a,b)$ is a piecewise
    quasinormal path of the form $\gamma = \gamma_\v$ where $\v=(v_0,v_1,\ldots,v_m)\in V_n(a,b)$. 
		We now try to relate $m$, the number of segments in the path, to
    $G_n(\gamma)$, the path energy.  Recall the formula \eqref{G_n_equation}.

Let $B_i$ denote the (open) Euclidean ball of radius
    $\epsilon_n / 2$ around $v_i$. We claim that
    $| B_i \cap \{v_0,\ldots,v_m\}| \leq 2$. To see this, suppose that
    there are at least $3$ points of $\{v_0,\ldots,v_m\}$ in $B_i$. Let $v_j$ and $v_l$ denote the points in $B_i$ with the smallest and largest index, respectively. 
     By minimality of $\gamma$, $v_j \neq v_l$. Let $v_k$ denote
    a third point in $B_i$.

    As $v_k, v_l\in B_i$, we have $|v_k - v_l| < \epsilon_n$, and so these points are connected in the graph. Applying the
    triangle inequality for $d_f$, valid when $p=1$ (cf. \eqref{p=1triangle}), and noting on the event $A_2$
    that $v_k \not \in S_{\gamma_{v_j,v_l}}$, we have
    \begin{equation*}
      d_f(v_j,v_l) < d_f(v_j,v_k) + d_f(v_k,v_l)
      \leq d_f(v_j,v_{j+1}) \ldots + d_f(v_{l-1},v_l).
    \end{equation*}
     Thus, the path
    $\tilde{\gamma} = \gamma_\w$, where $\w=(v_0,\ldots,v_j,v_l,\ldots,v_m)$,
    satisfies $G_n(\tilde{\gamma}) < G_n(\gamma)$. This contradicts
    the optimality of $\gamma$, and therefore
    $|B_i \cap \{v_0,\ldots,v_m\}| \leq 2$.

    We may thus cover the vertices of $\gamma$ with balls
    $\{B_i\}_{i=1}^m$, centered on the vertices $\{v_i\}_{i=1}^m$, and
    each of these balls contains at most two vertices. It follows that
    there is a subcover by $s\geq m/2$ balls, $\{B'_1,\ldots,B'_s\}$, no two of them containing a common point in $\v$.

    A lower bound for $G_n(\gamma)$ is found by considering that
    part of the $G_n$-integral contributed to by the portion of the
    path $\gamma$ in $B'_i$. Each such portion, if it does not
    terminate in $B'_i$, must visit both the center of $B'_i$ and the
    boundary $\partial B'_i$, and hence has Euclidean length at least
    $\epsilon_n / 2$. Summing over these portions, we obtain
    \begin{equation*}
      \frac{m\epsilon_n}{4} \leq \sum_{i=1}^s \frac{\epsilon_n}{2} \leq L,
    \end{equation*}
    where $L = \int_0^1 |\dot{\gamma}(t)|\, dt$ is the Euclidean
    arclength of $\gamma$.

    By \eqref{fbounds}, it follows that
    \begin{equation} \label{mepsbound}
      m_1 \frac{m\epsilon_n}{4} \leq  m_1L \leq \int_0^t f(\gamma(t), \dot\gamma(t))dt = G_n(\gamma).
    \end{equation}

    To get a Lipschitz bound for $\gamma$, recall the bound
    \eqref{uniformlipschitz}.  Then, $\gamma_{v_{i-1},v_i}(m\cdot -i)$ is
    Lipschitz with constant $(m_2/m_1)m|v_{i-1} - v_i|$.  It follows
    that $\gamma$, being the concatenation of these segments, 
    satisfies
    \begin{eqnarray*}
      |\gamma(s) - \gamma(t)| &=& \big| \sum_{r=0}^{q-1} \gamma(w_r)-\gamma(w_{r+1})\big|\\
      &\leq&
      (m_2/m_1) m \max\left(|v_0-v_1|,\ldots,|v_{m-1} - v_m|\right)\sum_{r=0}^{q-1}|w_r - w_{r+1}|\\
      &\leq& (m_2/m_1) m \epsilon_n|s-t|, 
    \end{eqnarray*}
        where $s=w_0<\cdots w_q=t$, $\{w_r\}_{r=1}^{q-1}\subset \{j/m\}_{j=1}^{m-1}$ so that $|w_r - w_{r+1}|\leq 1/m$ for $0\leq r\leq q-1$.
    Then, with \eqref{mepsbound}, we obtain
    \begin{equation*}
      |\gamma(t) - \gamma(s)| \leq (m_2/m_1) m \epsilon_n|s-t|
      \leq (4/m_1)(m_2/m_1) G_n(\gamma)|s-t|,
    \end{equation*}
		finishing the proof.
		\end{proof}

  We now prove our compactness property.

\begin{lemma}[Compactness Property]\label{compactness}
\indent
  
  (I).  Suppose for all large $n$ that either $\gamma_n\in \argmin L_n$ or $\gamma_n \in \argmin G_n$.   Then, for realizations $\{X_i\}$ in the probability $1$ set $A_1$, we have $\sup_n F(\gamma_n)<\infty$.  
  
(II).  Consider now the following cases:
  
\noindent (i) Suppose paths  $\gamma_n\in \argmin G_n$ for all large $n$.
  
  \noindent (ii) Suppose $p > 1$, and paths $\gamma_n\in \Omega(a,b)$ for all large $n$ such that $\sup_n F(\gamma_n) < \infty$.

  Then, in case (i) when $p>1$, and in case (ii),
  with respect to realizations $\{X_i\}$ in the probability $1$ set $A_1$, we have $\{\gamma_n\}$
  is relatively compact for the topology of uniform convergence.  For case (i) when $d\geq 2$ and $p=1$, the same conclusion holds with respect to realizations $\{X_i\}$ in the probability $1$ set $A_1\cap A_2$.
\end{lemma}

\begin{proof}
  We first prove the bound $\sup_n F(\gamma_n)<\infty$ in part (I). Choose a
  $\tilde{\gamma} \in \Omega(a,b)$, where \eqref{growthbound} holds, and $F(\tilde{\gamma}) <
  \infty$. By Lemma \ref{limsup}, there is a sequence
  $\{\tilde{\gamma}_n\}$ of either piecewise linear or quasinormal paths such that
  $\limsup_{n \to \infty} F(\tilde{\gamma}_n) \leq
  F(\tilde{\gamma})$. Hence, by minimality of $\{\gamma_n\}$, with respect to paths in either $\Omega_n^l$ or $\Omega_n^\gamma$, we have
  \begin{equation} \label{fnbound}
    \sup_n F(\gamma_n) \leq \sup_n F(\tilde{\gamma}_n) < \infty.
  \end{equation}
  
  We now argue the claims for cases (i) and (ii).  In both cases, as $D$ is bounded, the paths $\gamma_n : [0,1] \to D$ are
  uniformly bounded. To invoke the Arzela-Ascoli theorem, we
  must show that $\{\gamma_n\}$ is an equicontinuous family.

  In case (i), when $d\geq 2$ and $p=1$,  by Lemma \ref{p=1compactnesslemma} on realizations in $A_1\cap A_2$, we have
  $|\gamma_n(s) - \gamma_n(t)| \leq C G_n(\gamma_n) |s-t|$, with $C$
  independent of $n$. As $G_n(\gamma_n)=F(\gamma_n)$, combining with \eqref{fnbound}, it follows that
  $\{\gamma_n\}$ is equicontinuous.
  
  If $p > 1$, with respect to realizations in $A_1$, \eqref{fnbound} implies that, if case (i) holds for the sequence,
  then case (ii) holds. 
  
  Without loss of generality, then, we focus our
  attention now on case (ii).
  Recall, by \eqref{fbounds}, that $m_1 |v|^p \leq f(x,v)$.  Let $q$ be the conjugate of $p$, that is $1/p + 1/q =1$.  Then, for $0 \leq s < t \leq 1$,
  \begin{equation} \label{equicontinuitybound}
  \begin{aligned} 
    |\gamma_n(s) - \gamma_n(t)| & \leq \int_s^t |\gamma_n'(r)|\, dr \\
                      & \leq |t-s|^{1/q} \left(\int_0^1 |\gamma_n'(t)|^p \, dt\right)^{1/p} \\
                      & \leq (|t-s|^{1/q} / m_1^{1/p}) \left(\int_0^1 f(\gamma_n(t),\gamma_n'(t))\, dt\right)^{1/p} \\
    &= (|t-s|^{1/q} / m_1^{1/p}) F(\gamma_n)^{1/p}
  \end{aligned}
  \end{equation}

  Combining \eqref{equicontinuitybound} and the assumption in case (ii) that $\sup_n F(\gamma_n)<\infty$,  we have
  $|\gamma_n(s) - \gamma_n(t)| \leq C |t-s|^{1/q}$ for a constant $C$
  independent of $n$, and hence $\{\gamma_n\}$ is equicontinuous.
 \end{proof}

\subsection{Proof of Theorems \ref{theorem1} and \ref{theorem2}}

With the preceding `Gamma convergence' ingredients in place, the proofs of Theorem \ref{theorem1} and \ref{theorem2} are similar, and will be given together.

\medskip
\noindent {\bfseries Proofs of Theorems \ref{theorem1} and \ref{theorem2}}.
  Fix a realization $\{X_i\}$ in the probability $1$ set $A_1$.
  Let $\{\gamma_n \}$ be a sequence of paths such that, for all large $n$, we have either $\gamma_n \in \argmin L_n$ or $\gamma_n\in \argmin G_n$. Supposing
  that $\{ \gamma_n \}$ has a subsequential limit
  $\lim_{k \to \infty} \gamma_{n_k} = \gamma$, with respect to the topology
  of uniform convergence, we now argue that $\gamma \in \argmin F$.

  By the `liminf' Lemma \ref{liminf}, we have
  \begin{equation} \label{lower2}
    F(\gamma) \leq \liminf_{k \to \infty} F(\gamma_{n_k}).
  \end{equation}
  Let $\gamma^* \in \argmin F$. Then, by inequality \eqref{derivbounds} of Proposition \ref{existence}, there exist constants $c_1,c_2$ such that $c_1|s-t| \leq |\gamma^*(s) - \gamma^*(t)| \leq c_2 |s-t|$ for $0 \leq s,t \leq 1$. Hence, by the `limsup' Lemma \ref{limsup}, there exists a sequence $\{\gamma^*_n\}$, of either piecewise linear or quasinormal paths, converging uniformly to $\gamma^*$ and $\limsup_{n \to \infty} F(\gamma_n^*) \leq F(\gamma^*)$. 
  Recall that $F(\gamma) = L_n(\gamma)$ or $F(\gamma)=G_n(\gamma)$ when $\gamma$ is piecewise linear or quasinormal respectively.  Combining with \eqref{lower2} and minimality of $\gamma_{n_k}$, we have
  \begin{equation} \label{squeeze2}
    F(\gamma) \leq \liminf_{k \to \infty} F(\gamma_{n_k}) \leq \limsup_{k \to \infty} F(\gamma^*_{n_k}) \leq F(\gamma^*)=\min F,
  \end{equation}
  and so $\gamma \in \argmin F$. 
  
  In the case the paths $\{\gamma_n\}$ are piecewise linear, since $F(\gamma_n) = \min L_n$,
   it follows from \eqref{squeeze2} that
  \begin{equation*} \label{minconverge}
    \lim_{k \to \infty} \min L_{n_k}= \min F.
  \end{equation*}
Similarly, when $\{\gamma_n\}$ are piecewise quasinormal, $\lim_{k\rightarrow\infty} \min G_{n_k} = \min F$.

  Therefore, we have shown that, if a subsequential limit of $\{ \gamma_n \}$ exists, it is an optimal continuum path $\gamma \in \argmin F$.  
  
  Consider now Theorem \ref{theorem1}, where $p>1$ and $\gamma_n\in \argmin L_n$, and part (1) of Theorem \ref{theorem2} where $p>1$ and $\gamma_n \in \argmin G_n$.  By the `compactness' Lemma \ref{compactness}, $\sup_n F(\gamma_n)<\infty$ and a subsequential limit exists.    
  
Consider now part (2) of Theorem \ref{theorem2} where $d\geq 2$, $p=1$ and $\gamma_n\in \argmin G_n$. Suppose that the realization $\{X_i\}$ belongs also to the probability $1$ set $A_2$.  Then, subsequential limits follow again from the `compactness' Lemma \ref{compactness}. 

  Now, consider any subsequence
  $\{n_k\}$ of $\mathbb{N}$. Then, by the work above, applied to the sequence $\{n_k\}$, there is a further subsequence $\{n_{k_j}\}$, and
  a $\gamma \in \argmin F$, with $\gamma_{n_{k_j}} \to \gamma$ uniformly, in the settings of Theorems \ref{theorem1} and \ref{theorem2}.
	Moreover,
  $\lim_{j \to \infty} \min L_{n_{k_j}} = \min F$ when the paths $\{\gamma_{n_{k_j}}\}$ are piecewise linear, and $\lim_{j\to \infty} \min G_{n_{k_j}} = \min F$ when the paths are piecewise quasinormal. 
	
	Since this argument is valid for any subsequence $\{n_k \}$ of $\mathbb{N}$, we recover that $\min L_n \to \min F$ or $\min G_n \to \min F$ when respectively the paths are piecewise linear or quasinormal.
Finally, if $F$ has a unique minimizer $\gamma$, by considering subsequences again, the whole sequence $\{\gamma_n\}$ must converges uniformly to $\gamma$.
\qed

\subsection{Proof of Corollary \ref{hausdorffcor}}

Corollary \ref{hausdorffcor} is a statement about Hausdorff
convergence. In order to adapt the results of Theorems \ref{theorem1} and \ref{theorem2} to this end, we make the following observation.

\begin{prop} \label{unifhausdorff} Fix a realization
  $\{X_i\}$ in $A_1$, and consider a sequence of paths $\{\gamma_n\}$
  such that $\gamma_n$ for all large $n$ is either in the form (1) $\gamma_n = \gamma_{\v^{(n)}}$ or (2) $\gamma_n = l_{\v^{(n)}}$, where
  $\v^{(n)} \in V_n(a,b)$.  Suppose that $\gamma_n$
  converges uniformly to a limit $\gamma \in \Omega(a,b)$. Then,
  \begin{equation*}
    \lim_{n \to \infty} d_{haus}(\v^{(n)}, S_{\gamma}) = 0.
  \end{equation*}
\end{prop}

\begin{proof} Write $\v^{(n)} = (v_0^{(n)}, \ldots, v_{m(n)}^{(n)})$.
  Since $\gamma_{n} \to \gamma$ uniformly and
  $v_i^{(n)} = \gamma_n(\frac{i}{m(n)})$, it follows that
\begin{equation} \label{hausdorff1}
  \lim_{n \to \infty} \max_{1 \leq i \leq m(n)} \inf_{x \in S_{\gamma}} |v_i^{(n)} - x| = 0.
\end{equation}

On the other hand, consider $t$ with $i/m(n) \leq t <
(i+1)/m(n)$. 
In case (1), 
$$|\gamma(t) - v_i^{(n)}| \leq |\gamma(t) - \gamma(i/m(n))| + \|\gamma_n - \gamma\|_\infty \leq C/m(n) + \|\gamma_n - \gamma\|_\infty,$$
where $C$ is the Lipschitz constant of $\gamma$.
In case (2), using linearity of the path, $|\gamma(t) - v_i^{(n)}|\leq |v^{(n)}_{i+1} - v^{(n)}_i| + \|\gamma_n - \gamma\|_\infty \leq \epsilon_n + \|\gamma_n - \gamma\|_\infty$. 

Since $\v^{(n)}$ is a path in $V_n(a,b)$, one may bound $\epsilon_n m(n) \geq \sum_{i=0}^{m(n)} |v_i^{(n)} - v_{i+1}^{(n)}| \geq |a-b|$, and so $m(n) \geq |a-b|/\epsilon_n$ diverges.  Hence, in both cases,
\begin{equation} \label{hausdorff2}
  \lim_{n \to \infty} \sup_{x \in S_{\gamma}} \min_{1 \leq i \leq m(n)} |x - v_i^{(n)}| = 0.
\end{equation}
Combining \eqref{hausdorff1} and \eqref{hausdorff2}, it follows that
$\lim_{n \to \infty} d_{haus}(\v^{(n)},S_{\gamma}) = 0$.
\end{proof}

We now proceed to prove Corollary \ref{hausdorffcor}.

\medskip \noindent {\bfseries Proof of Corollary \ref{hausdorffcor}}.   We give the argument for the case of piecewise linear optimizers, as the the argument is exactly the same for piecewise quasinormal paths, using Theorem \ref{theorem2} instead of Theorem \ref{theorem1} below.
 
  Suppose $l_n=l_{\v^{(n)}}\in \argmin L_n$  is a sequence of paths where $\v^{(n)}\in V_n(a,b)$.
	  By Theorem
  \ref{theorem1}, with respect to a probability $1$ set of realizations $\{X_i\}$, any subsequence of $\{l_n\}$ has a further subsequence $\{l_{n_k}\}$ which converges uniformly to a $\gamma\in \argmin F$. 
  By Proposition
  \ref{unifhausdorff}, it follows that
  $\lim_{k \to \infty} d_{haus}(\v^{(n_k)},S_{\gamma}) = 0$.

Suppose now that $F$ has a unique (up to reparametrization) minimizer $\gamma$.  Note that $S_\gamma$ is invariant under reparametrization of $\gamma$.  Then, we conclude that all limit points of $\{v^{(n)}\}$ correspond to $S_\gamma$, and hence the whole sequence $v^{(n)}$ converges to $S_\gamma$,
  $\lim_{n \to \infty} d_{haus}(v^{(n)},S_{\gamma}) = 0$.
\qed

\section{Proof of Theorems \ref{p>1approxthm}, \ref{p=1approxthm} and \ref{linear_p=1approxthm}}
\label{riemann_proofs}

The proofs of Theorems \ref{p>1approxthm}, \ref{p=1approxthm} and \ref{linear_p=1approxthm} all make use of Theorems \ref{theorem1} and \ref{theorem2} in comparing the costs $H_n$ and $L_n$ to $G_n$.  When $p=1$, as with respect to Theorem \ref{theorem2}, the arguments in Theorems \ref{p=1approxthm} and \ref{linear_p=1approxthm} are more involved, especially with respect to the minimal cost $H_n$-path convergence, where several geometric estimates are used to show a compactness principle.

We begin with the following useful fact.
\begin{prop} \label{minprop}
  Suppose $U, W, C : X \to \mathbb{R}$ are 
  functions such that $|U(x) - W(x)| \leq C(x)$ for all
  $x \in X$. If $U(x_1) = \min U$ and $W(x_2) = \min W$ then
  $$-C(x_1)\leq \min U - \min W \leq C(x_2).$$
\end{prop}
\begin{proof}
  For any $y \in X$ we have
  $\min U = U(x_1) \leq U(y) \leq W(y) + C(y)$. Taking $y = x_2$ gives
  $\min U \leq \min W + C(x_2)$.  The other inequality follows
  similarly.
\end{proof}

\subsection{Proof of Theorem \ref{p>1approxthm}}
  Suppose $\v = (v_0,v_1,\ldots,v_m) \in V_n(a,b)$. Then (Lip) implies, for $0\leq i\leq m-1$, that
  \begin{equation} \label{linpathineq}
     \Big| \int_0^1 f(v_i + t(v_{i+1}-v_i), v_{i+1}-v_i) \, dt - f(v_i, v_{i+1}-v_i)\Big| 
      \leq c |v_{i+1} - v_i|^{p+1}.  
  \end{equation}
  Since $v_i,v_{i+1}$ are neighbors in the $\epsilon_n$-graph, $|v_{i+1}-v_i| \leq \epsilon_n$. Thus, from the homogeneity \eqref{fhomo} and bounds \eqref{fbounds} of $f$, rescaling
  \eqref{linpathineq} gives
  \begin{eqnarray*}
	&&\big |\int_{\frac{i}{m}}^{\frac{i+1}{m}} f(v_i + m(s-\frac{i}{m}) (v_{i+1} - v_i),m(v_{i+1} -
v_i)) \, ds - \frac{1}{m}f(v_i,m(v_{i+1} - v_i))\big|\\ 
&&\ \ \ \ \ \leq
\frac{c\epsilon_n}{m} |v_{i+1} - v_i|^p m^p.
\end{eqnarray*} 
Recall formulas \eqref{L_n_equation} and \eqref{H_n_eq}.  Summing over $i$ gives the
following estimate relating $L_n$ and $H_n$:
  \begin{equation} \label{fgcompare0}
      |L_n(l_\v) - H_n(\v)| \leq \frac{c \epsilon_n}{m} \sum_{i=0}^{m-1} |v_{i+1} - v_i|^p m^p.
      \end{equation}
Applying \eqref{fbounds}, the right-side of \eqref{fgcompare0} can be bounded above in terms of both $L_n(l_\v)$ and $H_n(\v)$.  Hence, with $c'=cm_1^{-1}$,
\begin{equation}\label{fgcompare}
   |L_n(l_\v) - H_n(\v)|    \leq c' \epsilon_n \min(L_n(l_\v), H_n(\v)).
  \end{equation}

  Suppose $l_{\v^{(n)}}\in \argmin L_n$ and $\w^{(n)}\in \argmin H_n$.  An immediate consequence of \eqref{fgcompare} and Proposition \ref{minprop} is 
  \begin{eqnarray*}
    -c'\epsilon_n \min L_n &\leq& -c'\epsilon_n \min (L_n(l_{\v^{(n)}}), H_n(\v^{(n)}))\\
		&\leq &\min L_n - \min H_n\\
    & \leq& c' \epsilon_n \min(L_n(l_{\w^{(n)}}), H_n(\w^{(n)})) \leq c'\epsilon_n \min  H_n.
  \end{eqnarray*}

  By Theorem \ref{theorem1}, we have $\lim_{n \to \infty} \min L_n  = \min F$ for almost all realizations $\{X_i\}$ (those in $A_1$ as the proof shows).  Then, $\min H_n \leq (1 + c'\epsilon_n)\min L_n$ and so $\limsup \min H_n\leq \min F$ a.s.  In particular, as $\min F <\infty$, we have $\sup_n \min H_n <\infty$ a.s.  
  
  On the other hand, $\min H_n \geq \min L_n -c'\epsilon_n \min H_n$ a.s.  As $\sup_n \min H_n<\infty$, we observe that
  $\liminf \min H_n \geq \min F$ a.s.  Hence, $\min H_n \to \min F$ a.s.  
  This finishes one part of Theorem \ref{p>1approxthm}.
  
To address the others, consider $l_{\w^{(n)}}$, the piecewise linear interpolation of $\w^{(n)}$. 
By \eqref{fgcompare}, we have
  \[ L_n({l}_{\w^{(n)}}) \leq (1+c' \epsilon_n) H_n(\w^{(n)}) = (1+c'\epsilon_n)\min H_n. \]

    Moreover, noting the optimality of $l_{\v^{(n)}}$ and $\w^{(n)}$ gives
  \begin{equation}
      L_n(l_{\v^{(n)}}) \leq L_n(l_{\w^{(n)}}) 
      \leq (c' \epsilon_n +1) H_n(\v^{(n)}).
		\label{thm3eq}
  \end{equation}
  Another application of \eqref{fgcompare} yields 
  $$H_n(\v^{(n)}) \leq c'\epsilon_n L_n(l_{\v^{(n)}}) + L_n(l_{\v^{(n)}}).$$  
  Hence, the left-side of \eqref{thm3eq} is bounded as
  \begin{eqnarray}
   \min L_n =  L_n(l_{\v^{(n)}}) &\leq& L_n(l_{\w^{(n)}}) \leq (c'\epsilon_n+1)^2 L_n(\v^{(n)})\nonumber\\
    &=& (c'\epsilon_n +1)^2 \min L_n.
		\label{samelimit0}
      \end{eqnarray}

  Hence, as $\min L_n \to \min F$ a.s., we have 
  \begin{equation} \label{samelimit} 
  \min F =
    \lim_{n \to \infty} L_n(l_{\v^{(n)}}) =\lim_{n\to\infty} L_n(l_{\w^{(n)}}) \ \ \text{a.s.} \end{equation} 
 We also observe, as a consequence, that $\sup_n L_n(l_{\w^{(n)}}) < \infty$ a.s. 
 
  Given that $p>1$, by the `compactness' Lemma
  \ref{compactness}, with respect to realizations $\{X_i\}$ in the probability $1$ set $A_1$, any subsequence of $\{l_{\w^{(n)}}\}$ has a further uniformly convergent subsequence to a limit 
  $\tilde{\gamma} \in \Omega(a,b)$. By the `liminf' Lemma
  \ref{liminf},
  $F(\tilde{\gamma}) \leq \liminf_{n \to \infty} L_n(l_{\w^{(n)}})$ a.s. Finally, by \eqref{samelimit}, it follows that $F(\tilde \gamma) = \min F$ and so $\tilde \gamma\in \argmin F$.
	Consequently, if $F$ has a unique minimizer $\gamma$, then the whole sequence $\{l_{\w^{(n)}}\}$ converges uniformly almost surely to it.

The proofs of statements about Hausdorff convergence follow the same arguments as given for Corollary \ref{hausdorffcor}, and are omitted. 
\qed

\subsection{Proof of Theorems \ref{p=1approxthm} and \ref{linear_p=1approxthm}}
We prove Theorems \ref{p=1approxthm} and \ref{linear_p=1approxthm} in two parts.

\medskip
\noindent {\bfseries Proof of Theorems \ref{p=1approxthm} and \ref{linear_p=1approxthm}.}
First, we prove in Proposition \ref{p=1_minapproxthm} that the minimal costs of $H_n$ and $L_n$ converge to $\min F$, making use of comparisions with quasinormal paths, for which we have control in Theorem \ref{theorem2}.  

 Second, in Proposition \ref{p=1_convergethm} in Subsection \ref{convergence_subsect}, we show that the minimizing paths converge in the various senses desired.  A main tool in this proof is a compactness property (Proposition \ref{gcompactness}), for minimal $H_n$ and $L_n$-paths when $p=1$, shown in Subsection \ref{p=1compactness_section}.  \qed

\vskip .1cm
To supply the proofs of the desired propositions, we now obtain an useful estimate between the cost of a quasinormal path and a linear one.

\begin{prop}
\label{quasilinearcompare}
Suppose 
$d\geq 2$, $p=1$, and that $f$ also satisfies (Lip) and (Hilb).  For $a,b\in D$ such that $|b-a|\leq 1$, there is a constant $c_1$ such that
$$\big|d_f(a,b) - f(a,b-a)\big| \leq c_1|b-a|^2.$$
In particular, as $d_f(a,b) = \int_0^1 f(\gamma(t), \dot\gamma(t))dt$ for the quasinormal path $\gamma = \gamma_{a,b}$ connecting $a$ and $b$, 
we have
$$\Big| \int_0^1 f(\gamma(t),\dot\gamma(t))dt - f(a,b-a)\Big| \leq c_1|b-a|^2.$$
\end{prop}

\begin{proof}    
 By \eqref{fbounds}, for a Lipschitz path $\beta$ from $a$ to $b$, we have
 $$m_1|\int_0^1|\dot\beta(t)|dt \leq \int_0^1 f(\beta(t), \dot\beta(t))dt \leq m_2\int_0^1 |\dot\beta(t)|dt.$$
 Optimizing over $\beta$, we recover that $m_1|b-a| \leq \int_0^1 f(\gamma(t), \dot\gamma(t))dt \leq m_2|b-a|$.  By \eqref{fbounds} again, we have that the arclength of $\gamma$ satisfies $ \int_0^1 |\dot \gamma(t)|dt \leq (m_2/m_1)|b-a|$.    
In particular, the path $\gamma$ is constrained in the Euclidean ball $B$ around $a$ of radius $(m_2/m_1)|b-a|$.  Note also that the minimizing Euclidean path $\widetilde\gamma$, with constant speed $|b-a|$ on the straight line from $a$ to $b$ in times $0\leq t\leq 1$, is also constrained in this ball.

Now, for a Lipschitz path $\beta$, constrained in the ball $B$, expand 
$$\int_0^1 f(\beta(t),\dot\beta(t))dt = \int_0^1 f(a, \dot\beta(t))dt + \int_0^1 \big(f(\beta(t), \dot\beta(t))-f(a, \dot\beta(t))dt.$$
As the paths are in $B$, by (Lip), with respect to a Lipschitz constant $C$,
\begin{align*}
  |f(\beta(t),\dot\beta(t)) - f(a,\dot\beta(t))| &\leq C |\beta(t)-a||\dot\beta(t)| \\
  & \leq C(m_2/m_1)|b-a||\dot\beta(t)|.
\end{align*}
Therefore, with respect to Lipschitz paths $\beta$ constrained in $B$,
$$\Big|\int_0^1 f(\beta(t),\dot\beta(t))dt - \int_0^1 f(a,\dot\beta(t))dt\Big| \leq C(m_2/m_1)|b-a|\int_0^1|\dot\beta(t)|dt.$$
Note, by condition (Hilb) that, for the cost with respect to $f(a,\cdot)$, straight lines are geodesics, and in particular $\widetilde{\gamma}(t)= (1-t)a + tb$ is optimal.  Hence, the minimal $F$-cost, with respect to $f(a,\cdot)$, of moving from $a$ to $b$, given invariance to parametrization when $p=1$, is $f(a,b-a)$.

Then, by Proposition \ref{minprop}, applied to the two functionals of $\beta$ on the left-hand sides,
we obtain 
\begin{eqnarray*}
\big|d_f(a,b) - f(a,b-a)\big| &\leq& 
C(m_2/m_1)|b-a|\max \Big[\int_0^1 |\dot\gamma(t)|dt, \int_0^1 |\dot{\widetilde{\gamma}}(t)|dt\Big],\\
&\leq& C(m_2/m_1)^2|b-a|^2,
\end{eqnarray*}
noting the arclength bounds of $\gamma=\gamma_{a,b}$ and $\widetilde \gamma$ above.
\end{proof}

\begin{prop}
\label{p=1_minapproxthm}
  Suppose 
	$d\geq 2$, $p=1$, and that $f$ also satisfies (Lip) and (Hilb).
	With respect to realizations $\{ X_i\}$ in the
  probability $1$ set $A_1\cap A_2$, the minimum values of
  $H_n$ and $L_n$ converge to the minimum of $F$,
  \[ \lim_{n \to \infty} \min_{v \in V_n(a,b)} H_n(v) = \lim_{n\to\infty}\min_{\gamma\in \Omega^l_n(a,b)}L_n(\gamma)= \min_{\gamma \in \Omega(a,b)} F(\gamma). \]
\end{prop}

\begin{proof}
  Consider the energies $G_n$ and $H_n$ in \eqref{G_n_equation} and \eqref{H_n_eq}.
	For $\gamma = \gamma_{\v}$, the piecewise quasinormal path through the vertices $\v=(v_0,v_1,\ldots,v_m)\in V_n(a,b)$, we have, noting $p=1$, that
  $$
	G_n(\gamma) \ = \ 
    \sum_{i=0}^{m-1} \int_{0}^1 f(\gamma_i(t),\dot\gamma_i(t))\, dt, 
	$$
  where $\gamma_i =\gamma_{v_i, v_{i+1}}$ is a quasinormal path from $v_i$ to $v_{i+1}$.

  An application of Proposition \ref{quasilinearcompare}, noting that $|v_{i+1}-v_i|\leq \epsilon_n$, gives
  $$\Big| \int_0^1 f(\gamma_i(t),\dot{\gamma_i}(t))\, dt - f(v_i,v_{i+1}-v_i)\Big|
  \leq c_1 \epsilon_n |v_{i+1}-v_i|.$$ Summing this over $i$ gives
  \begin{equation} \label{fgquasibound}
    \begin{aligned}
      |G_n(\gamma) - H_n(\v)| &\leq c_1 \epsilon_n \sum_{i=0}^{m-1} |v_{i+1} - v_i|  \\
      & \leq c_1m_1^{-1} \epsilon_n \min(G_n(\gamma),H_n(\v)),
    \end{aligned}
  \end{equation}
  where the last inequality follows from applying \eqref{fbounds} to
  both $G_n$ and $H_n$.
	
	Recall the energy $L_n$ in \eqref{L_n_equation}. Similarly, and more directly, using (Lip), we have for a linear path $l = l_\v\in \Omega^l_n(a,b)$ through vertices $\v= (v_0,\ldots, v_m)\in V_n(a,b)$ that
	$$\Big| \int_0^1 f(l_i(t),\dot{l_i}(t))\, dt - f(v_i,v_{i+1}-v_i)\Big|
  \leq c_1 \epsilon_n |v_{i+1}-v_i|,$$ 
where $l_i = l_{v_i,v_{i+1}}$ is the linear path from $v_i$ to $v_{i+1}$ with slope $v_{i+1}-v_i$.
Summing over $i$, using \eqref{fbounds}, we obtain
\begin{equation}
\label{l_eq}
 |L_n(l) - H_n(\v)| \leq c_1 \epsilon_n \sum_{i=0}^{m-1} |v_{i+1} - v_i|  \leq c_1m_1^{-1} \epsilon_n \min(L_n(l),H_n(\v)).
\end{equation}

  We now reprise some of the argument for Theorem \ref{p>1approxthm}.  A consequence of \eqref{fgquasibound} and Proposition \ref{minprop} is
  \begin{equation*}
    -c_1m_1^{-1}\epsilon_n \min G_n \leq \min G_n - \min H_n \leq c_1m_1^{-1} \epsilon_n \min H_n.
  \end{equation*}
   Hence, $\sup_n \min H_n < \sup_n (1+c_1m_1^{-1}\epsilon_n)\min G_n$. 
  By Theorem \ref{theorem2}, as seen in its proof, for realizations in the probablility $1$ set $A_1\cap A_2$, we have
  $\lim_{n \to \infty} \min G_n = \min F$, which is finite.  Then, we conclude that also
  $\lim_{n \to \infty} \min H_n = \min F$ a.s.
	
	Now, we can repeat this same argument with $L_n$ and \eqref{l_eq} in place of $G_n$ and \eqref{fgquasibound}, using now $\min H_n\rightarrow \min F$ a.s., to conclude that also $\min L_n$ converges to $\min F$ a.s.
\end{proof}

\subsubsection{Compactness Property}
\label{p=1compactness_section}

When $d\geq 2$ and $p=1$, analogous to Lemma \ref{p=1compactnesslemma}, we formulate now a compactness property for minimal paths $w^{(n)} \in \argmin H_n$ and $l_{\v^{(n)}}\in \argmin L_n$.

It will be useful to consider a partition of $D$ by a regular
grid. Let $z \in \mathbb{Z}^d$ and let $\Box_{n,z}$ be the
intersection of the box $\prod_{i=1}^d [z_i \tau_n, (z_{i}+1)\tau_n)$
with $D$, where $\tau_n = \epsilon_n / \sqrt{d}$. We will refer to these sets as `boxes', with the understanding that the boundary
of $D$ results in some of these being irregularly shaped. Regardless,
each $\Box_{n,z}$ has diameter at most $\epsilon_n$, and so points of
$\{X_i\}_{i=1}^n$ in $\Box_{n,z}$ are all connected in the random geometric graph.

\begin{prop} \label{gcompactness}  Consider the assumptions in the second parts of Theorems \ref{p=1approxthm} and \ref{linear_p=1approxthm}. 
	 Suppose $\w^{(n)}\in \argmin H_n$, and consider the piecewise linear interpolations $l_n=l_{\w^{(n)}}$.
	Then, with respect to a realizations $\{X_i\}_{i\geq 1}$ in a probability $1$ subset of $A_1\cap A_2$, the sequence $\{l_n\}$ is relatively compact for the topology of uniform convergence.
	
Suppose now $l_{\v^{(n)}}\in\argmin L_n$.  Then, the same conclusion holds for the optimal linear interpolations $\{l_{\v^{(n)}}\}$.
\end{prop}

\begin{proof}
  We show that the sequence $\{ l_n \}$ is
  equicontinuous for almost all realizations in $A_1$. As the paths belong to a bounded set $D$, the proposition would then follow from the Arzela-Ascoli criterion.
  
  Partition $D$ by boxes
  $\{\Box_{n,z}\}_{z \in \mathbb{Z}^d}$. 
	By Lemma \ref{numvisited} below, the number of boxes
  visited by $\w^{(n)}$ and $\v^{(n)}$ is a.s. bounded by $C / \epsilon_n$ a.s., for all large $n$, where $C=C(f,d)$. By Lemma
  \ref{kbound} below, the number of vertices in $\w^{(n)}$ and $\v^{(n)}$ in a box is a.s. bounded by a constant
  $K=K(d, \rho, \alpha)$ for all large $n$. Thus, the maximum number $k_n$ of points in $\w^{(n)}$ and $\v^{(n)}$ is a.s. bounded,
  \[ k_n \leq K C /\epsilon_n. \] Since
  $|w^{(n)}_{i+1}-w^{(n)}_i|, |v^{(n)}_{i+1} - v^{(n)}_i| \leq \epsilon_n$, we obtain
  $\sup_{i} k_n |w^{(n)}_{i+1} - w^{(n)}_i| \leq KC$ and $\sup_i k_n |v^{(n)}_{i+1} - v^{(n)}_i| \leq KC$ a.s. for all large $n$.

  This implies a.s. that the piecewise linear
  paths $l_n$ and $l_{\v^{(n)}}$
  are Lipschitz, with
  respect to the fixed constant $KC$, for all large $n$, and so in
  particular equicontinuous. Indeed, for $l_n = l_{\w^{(n)}}$, where say $\w^{(n)}=(w^{(n)}_0,\ldots,w^{(n)}_{k_n})$, consider the part of the path
  connecting $w^{(n)}_i$ and $w^{(n)}_{i+1}$ from times $i/k_n$ to $(i+1)/k_n$,
  namely $l_n(t) = w^{(n)}_i(i+1 - k_n t) + w^{(n)}_{i+1}(k_n t - i)$. Then,
  we have $|\dot l_n(t)| = k_n |w^{(n)}_{i+1} - w^{(n)}_i|\leq k_n \epsilon_n \leq KC$.  The same argument holds for the paths $l_{\v^{(n)}}$.
\end{proof}

We now show the lemmas used in the proof Proposition \ref{gcompactness}.  We first bound the number of boxes visited by an optimal path.

\begin{lemma} \label{numvisited} 
Suppose $d\geq 2$, $p=1$, and that $f$ also satisfies 
(Lip) and (Hilb). Suppose $\w\in \argmin H_n$ and $l_{\v}\in \argmin L_n$ are optimal paths. Then, for realizations $\{X_i\}$ in the probability $1$ set $A_1\cap A_2$, for all large $n$, the number of distinct boxes
  $\{\Box_{n,z}\}_{z \in \mathbb{Z}^d}$ visited by $\w$ and $\v$ is bounded by
  $C / \epsilon_n$, where $C=C(d,f)$.
\end{lemma}
\begin{proof}
  Any visit of the path $\w$ or $\v$ to $2^d + 1$ distinct boxes has an Euclidean length
  of at least $\epsilon_n/\sqrt{d}$, since not all $2^d+1$ boxes can
  be adjacent. Recalling \eqref{fbounds}, and the formulas \eqref{L_n_equation} and \eqref{H_n_eq}, such a visitation therefore has a
  $H_n$ cost or $L_n$ cost of at least $m_1 \epsilon_n/\sqrt{d}$. So, we may bound
  the number of boxes visited by $\w$ or $\v$ by  $C' H_n(\w) / \epsilon_n$, where $C'=(2^d+1)\sqrt{d}/m_1$ depends on the dimension and $f$, but not on
  the path $\w$ or $\v$. Recalling Proposition \ref{p=1_minapproxthm}, we have with respect to realizations in $A_1\cap A_2$ that  
	$\lim \min H_n = \lim \min L_n = \min F<\infty$. The lemma then follows with say $C= 2C'\min F$.
\end{proof}

The next result shows that optimal paths $\w\in \argmin H_n$ and $l_{\v}\in \argmin L_n$ cannot have `long necks', and gives a bound on the number of points nearby an edge in the graph.

\begin{lemma} \label{initialwbounds} 
Suppose $d\geq 2$ and $p = 1$. Fix a realization $\{X_i\}$ in the probability $1$ set $A_1$.  Suppose $\w
\in \argmin H_n$ is an optimal path. If $i < j$ is such that
  $|w_i - w_j| < \epsilon_n$, then 
	\begin{equation}
	\label{w_k_eq}
	w_k \in B(w_i,C\epsilon_n), \text{for } i\leq k\leq j,
	\end{equation}
	where $C = 2(m_2 / m_1)$.

  Further, let $\Theta_n = \sup_{i,j:|w_i - w_j| < \epsilon_n} |i-j|$, and suppose $\epsilon_n = n^{-\delta}$, where
  $\delta > 1/(\beta + d)$ and $\beta > 0$. Then, with respect to realizations in a probability $1$ subset of $A_1$, for all large $n$, we have
	\begin{equation}
	\label{theta_eq}
	\Theta_n \leq \epsilon_n^{-\beta}.
	\end{equation}
	
	Suppose now $l_\v\in \argmin L_n$. The same conclusions \eqref{w_k_eq} and \eqref{theta_eq} hold with $\v$ in place of $\w$.
	
\end{lemma}
\begin{proof}
  We first show \eqref{w_k_eq}.
  If one of the points $\{w_k\}_{k=i}^{j}$ is more than an Euclidean distance $2(m_2/m_1) \epsilon_n$ away from $w_i$, then, recalling \eqref{fbounds}, we have 
	\begin{eqnarray*}
	\sum_{k=i}^{j-1}f(w_k,w_{k+1}-w_k)&\geq & m_1\sum_{k=i}^{j-1}|w_{k+1}-w_k|\\
	& \geq& 2m_2\epsilon_n\geq 2m_2|w_{j}-w_i| \ \geq\  2f(w_i,w_j-w_i).
	\end{eqnarray*}
		But, this implies that the path connecting $w_i$ and $w_j$ in one step would be less costly, with respect to $H_n$, than $\w$. Since $\w$ was taken to be minimal, all points $\{w_k\}_{k=i}^{j}$ therefore must belong to $B(w_i, 2(m_2/m_1)\epsilon_n)$.
		
		Suppose now $l_\v\in \argmin L_n$ and recall the form of $L_n$ when $p=1$ in \eqref{L_n_equation}.  Similarly, if one of the points $\{v_k\}_{k=i}^j$ is away from $v_i$ by $2(m_2/m_1)\epsilon_n$, we have
	\begin{eqnarray*}
	&&\sum_{k=i}^{j-1}\int_{0}^{1} f(l_{v_k,v_{k+1}}(t),v_{k+1}-v_k)dt \geq m_1\sum_{k=i}^{j-1}|v_{k+1}-v_k|\\
	& &\ \ \ \ \ \ \ \ \ \ \ \ \geq 2m_2\epsilon_n\geq 2m_2|v_{j}-v_i| \ \geq\  2\int_0^1 f(l_{v_i,v_j}(t),v_j-v_i)dt,
	\end{eqnarray*}	
	also a contradiction of minimality of $l_\v$.

  We now consider \eqref{theta_eq}.  The proof here is a count bound with respect to $\w$.  The argument with respect to $\v$ is exactly the same with $\v$ in place of $\w$.  
  
  First, $|j-i|$ is bounded by the number $N_{i,n}$ of points
  $\{X_i\}_{i=1}^n$, distinct from $w_i$, in the ball $B(w_i,C\epsilon_n)$.  Then,
  $N_{i,n}$ is $\mbox{Binomial}(n-1,p)$ where
  $p = \nu(B(X_i,C\epsilon_n))$. For $k\geq 1$, we have
    $$
      \mathbb{P}(N_{i,n} \geq k) \leq
			\binom{n-1}{k}p^k\ \leq \ \frac{(np)^k}{k!}.$$
    Recalling that $\nu = \rho dx$ and $\rho$ is bounded, we have
    $p \leq \|\rho\|_\infty {\rm Vol}(B(0,1)) \epsilon_n^d$, and so
    $P(N_{i,n} \geq k) \leq (C'n\epsilon_n^d)^k/{k!}$ for some constant $C'$.

 Let
    $N_n = \max \{N_{i,n}\}_{i=1}^n$. Then, a union bound gives that
    $$P(N_n \geq k) \leq \frac{n}{k!} \exp \{ k (\log n + d \log
    \epsilon_n + \log C') \}.$$

    Taking $k = \lceil \epsilon_n^{-\beta} \rceil$, and noting $k!\geq \sqrt{2\pi}e^{-k}k^{k+1/2}$, yields that
    \begin{equation} \label{Nntail}
      P(N_n \geq \epsilon_n^{-\beta}) \leq \frac{n}{\sqrt{2\pi} }\exp \{ \epsilon_n^{-\beta} (\log n + (\beta + d) \log \epsilon_n + \log C' +1) \}.
    \end{equation}
    If $\epsilon_n$ is in the form $\epsilon_n = n^{-\delta}$, then the right hand side of
    \eqref{Nntail} is summable when $(\beta + d)\delta > 1$. 
		
		Hence, by Borel-Cantelli lemma, for realizations in the intersection of a probability $1$ set and $A_1$ say, we have 
$\Theta_n \leq \max_i N_{i,n}\leq N_n\leq \epsilon_n^{-\beta}$ for all large $n$, and \eqref{theta_eq} follows. 
  \end{proof}

  We now give a lower bound on the cost of certain `long necks',
  that is the cost of an optimal $H_n$-path $\w$ of moving away from two close by vertices.
	
  \begin{lemma} \label{exitcost} 
	Suppose $d\geq 2$, $p = 1$ and that $f$ also satisfies
    $(TrIneq)$, and $(Pythag)$ with $\alpha > 1$.  Fix a realization $\{X_i\}$ in the probability $1$ set $A_1$.  Suppose
    $\w\in \argmin H_n$ is an optimal path, and let $i<j$ be indices such
    that $|w_i - w_j| < \epsilon_n$. Let $\l$ denote the straight line
    segment from $w_i$ to $w_j$.  
  Consider a neighborhood $\mathcal{A} = \cup_{x \in \l} B(x,r)$ of $\l$, with
  $r=\epsilon_n^\alpha$.  
	
	Then, if there is a point $w_k\not\in \mathcal{A}$ for
  some $i<k<j$, there is a constant $C=C(\alpha,f)$ such that
  \begin{equation} \label{show}
    \sum_{q=i}^{j-1} f(w_i,w_{q+1}-w_q) \geq f(w_i,w_{j}-w_i) + Cr^\alpha.
  \end{equation}
  
Suppose now $l_\v\in \argmin L_n$.  Then, \eqref{show} holds with $\v$ in place of $\w$.
\end{lemma}

\begin{proof}. The argument for $\v$ is the same as for $\w$, which we now present.
  Suppose a point $w_k$ is at least an Euclidean
  distance $r$ from $\l$.  By the
  (TrIneq) condition,
  $\sum_{q=i}^{j-1} f(w_i,w_{q+1}-w_q) \geq f(w_i,w_k - w_i) + f(w_i,
  w_j-w_k)$.

  By Lemma \ref{initialwbounds}, as $|w_i-w_j|<\epsilon_n$, we have
  $|w_kw_i| = |w_k-w_i|\leq 2(m_2/m_1)\epsilon_n$, which is strictly less than an $\eta<1$ for all large $n$. We also conclude
  $|w_jw_k|, |w_i w_j|<\eta<1$ for all large $n$.  In addition,
  $2(m_2/m_1)r^{1/\alpha}=2(m_2/m_1)\epsilon_n \geq |w_iw_j|$.  Thus,
  by (Pythag) with $x=w_i$, we obtain
  $f(w_i,w_k - w_i) + f(w_i, w_j-w_k)\geq f(w_i, w_j - w_i) +
  Cr^\alpha$, where $C=C(\alpha,f)$.

  Hence, \eqref{show} follows by combining the inequalities.
\end{proof}

We now bound the number of points of an optimal path in a box, the main estimate used in the proof of Proposition \ref{gcompactness}. 
The argument is in two steps.  In the first step, using a rough count on the number of vertices of the path within a given box, we may approximate the contribution to $H_n$ and $L_n$ from the vertices 
in the box in terms of a `localized' cost. Then, we use (Pythag), applied to the `localized' cost, to deduce that the optimal path in the box is trapped in a `small' set in the box.  The second step then is to show that such `small' sets contain only a constant number of points in $\{X_i\}$.

  \begin{lemma} \label{kbound}
	    Consider the assumptions of the second parts of Theorems \ref{p=1approxthm} and \ref{linear_p=1approxthm}.  
		Suppose
    $\w \in \argmin H_n$ is an optimal path. Then, with respect to realizations $\{X_i\}$ in a probability $1$ subset of $A_1$, for all large $n$,
    there is a constant $K$, 
		such that
    $| \{w_j\}_{j=1}^n \cap \Box_{n,z}| \leq K$
    for all
    $z \in \mathbb{Z}^d$.
    
    Suppose now $l_\v\in \argmin L_n$.  Then, the same statement above holds with $\v$ in place of $\w$.
  \end{lemma}

  \begin{proof}
    We will give the main argument for $\w$ and indicate modifications with respect to $\v$.  Consider a box $\Box \coloneqq \Box_{n,z}$. Boxes with at most one point trivially satisfy the claim in the lemma if say $K\geq 2$.  Suppose now that there are at least two points 
    in the box $\Box$.
		
		\vskip .1cm
		\noindent {\it Step 1.} Let 
		$w_i$ and $w_j$ be the first and
    last points of $\w$ in the box, that is, with the smallest and largest indices
    respectively.  By Lemma \ref{initialwbounds}, as $|w_i-w_j|<\epsilon_n$, we have
    $w_k \in B(w_i,C'\epsilon_n)$ for $i\leq k\leq j$.  
		Hence, by (Lip), we have
$    |\sum_{k=i}^{j-1} f(w_k,w_{k+1}-w_k) - \sum_{k=i}^{j-1} f(w_i,w_{k+1}-w_k)|
    \leq C' \epsilon_n^2 |j-i|$.
  Now, also by Lemma \ref{initialwbounds}, when $\delta> (\beta+d)^{-1}$ for $\beta>0$, we have
  $|j-i| \leq \epsilon_n^{-\beta}$.  Hence, the following estimate, with respect to a `localized' energy, where $x=w_i$ is fixed, is obtained:
  \begin{equation}
	\label{initial_eq}
	\Big |\sum_{k=i}^{j-1} f(w_k,w_{k+1}-w_k) - \sum_{k=i}^{j-1} f(w_i,w_{k+1}-w_k)\Big | \leq C' \epsilon_n^{2-\beta}.
	\end{equation}

Similarly, when $\v$ is considered, following the same reasoning using Lemma \ref{initialwbounds} and (Lip), we may obtain \eqref{initial_eq} with $\v$ in place of $\w$, and moreover
$$
\Big | \sum_{k=i}^{j-1} \int_0^1 f(l_{v_k, v_{k+1}}(t), v_{k+1}-v_k)dt - \sum_{k=i}^{j-1} f(v_k, v_{k+1}-v_k)\Big|
\leq C' \epsilon_n^2 |j-i| \leq C' \epsilon^{2-\beta}.
$$
Hence, combining these two estimates, we obtain that 
\begin{equation}
\label{linear_initial_eq}
\Big |\sum_{k=i}^{j-1} \int_0^1 f(l_{v_k, v_{k+1}}(t), v_{k+1}-v_k)dt - \sum_{k=i}^{j-1} f(v_i,v_{k+1}-v_k)\Big | \leq 2C' \epsilon_n^{2-\beta}.
\end{equation}

  Returning to the path $\w$, by Lemma \ref{exitcost}, noting \eqref{initial_eq}, any path $(w_i,\ldots,w_j)$ exiting
  $\mathcal{A}$, the $r=\epsilon_n^\alpha$-neighborhood of the line segment from $w_i$ to $w_j$,
  is costlier, with respect to $H_n$, than the $H_n$-cost $f(w_i,w_j-w_i)$ of a straight path
  connecting $w_i$ to $w_j$ in a single hop, as follows:
  \begin{equation}
  \label{H_n_tri_ineq}
  \sum_{q=i}^{j-1} f(w_q, w_{q+1}-w_q) - f(w_i, w_j-w_i) \geq Cr^\alpha - C'\epsilon_n^{2-\beta} = C\epsilon_n^{2\alpha} - C' \epsilon_n^{2-\beta}.
  \end{equation}
   
  Let us now consider $\v$.  Since $|v_i - v_j|\leq C\epsilon_n$ by Lemma \ref{initialwbounds}, using (Lip), we have that
  \begin{equation}
  \label{linear_onestep}
  \big| f(v_i, v_j-v_i) - \int_0^1 f(l_{v_i, v_j}(t), v_j-v_i)dt\big| \leq C'\epsilon_n^2.
  \end{equation}
   Following the same reasoning given with respect to $\w$, we may obtain \eqref{H_n_tri_ineq} with $\v$ in place of $\w$. Then, noting \eqref{linear_initial_eq},  
  a path $(v_i, \ldots, v_j)$, exiting the $r=\epsilon_n^\alpha$-neighborhood of the line segment from $v_i$ to $v_j$, has $L_n$ cost more than the one step $H_n$ cost $f(v_i, v_j-v_i)$ by the amount $Cr^\alpha - 2C'\epsilon_n^{2-\beta}$.  By \eqref{linear_onestep}, this $H_n$ cost differs from the one step $L_n$ cost $\int_0^1 f(l_{v_i, v_j}(t), v_j-v_i)dt$ of moving from $v_i$ to $v_j$ by $C'\epsilon_n^{2}$.
  
  Therefore, the cost savings of moving in one step, in considering $\w$ or $\v$ which exit the $r$-neighborhood, is bounded below by
  $C r^\alpha - 3C'\epsilon_n^{2-\beta} = O(\epsilon_n^{2\alpha} - \epsilon_n^{2-\beta})$, which is positive, for all large $n$, when $\alpha^2<2-\beta$.  This is the case when we fix $\beta = (2-\alpha^2)\eta>0$, for an $0<\eta<1$, since $1<\alpha<\sqrt{2}$. 
  
  Hence, with this choice of $\beta$, such exiting paths are not optimal, and all the points
  $\{w_i,\ldots, w_j\}$ or $\{v_i, \ldots, v_j\}$  in the box must belong to the $r=\epsilon_n^\alpha$-neighborhood of the line segment connecting the $i$th and $j$th points.

	 We note, given the value of $\beta$, to use Lemma \ref{initialwbounds} above, the exponent $\delta$ should satisfy
	$\delta> [(2-\alpha^2)\eta +d]^{-1}$, afforded by our assumptions.
	
\vskip .1cm

\noindent {\it Step 2.}. We now focus on $\w$ as the following counting argument is the same with respect to $\v$.
  We will count the points in the small set $\mathcal{A}$.  The cardinality
  $|\{w_{k}\}_{k=i}^j|=|j-i|$ is bounded by
  $| \mathcal{X}_n \cap \mathcal{A}| = 2 + N_{n,z}$, where $N_{n,z}$ is the Binomial$(n-2, \nu(\mathcal{A}))$ count of the number of points in $\mathcal{X}_n$ distinct from $w_i,w_j$ in the set $\mathcal{A}$.  Note, as $\mathcal{A}$ is nearly a cylinder
  with length $\epsilon_n$ and radius $\epsilon_n^\alpha$, and $\rho$ is bounded, we have that
  $\nu(\mathcal{A}) \leq C(\rho) \epsilon_n^{\alpha (d - 1) +1}$.  Then, 
$$P(N_{n,z} \geq K)\leq (n^K/K!) \nu(A)^{K} \leq C(\rho)^Kn^K\epsilon_n^{K\alpha d + K(1-\alpha)}.$$ 
Hence, by a union of events bound, as the number of boxes intersecting $D$ is bounded by $C'\epsilon_n^{-d}$,
we have
\begin{equation}
\label{count_prob}
P(\exists z \in \mathbb{Z}^d {\rm \ such\ that\ } N_{n,z} \geq K) \leq
C'C(\rho)^Kn^K\epsilon_n^{K\alpha d + K(1-\alpha)-d}.
\end{equation}

Suppose $\epsilon_n$ is of the form $\epsilon_n = n^{-\delta}$ for $0<\delta<1/d$.  If $d<K<\infty$ and
\begin{equation}
\label{delta_eq}
\delta> (K+1)/[K(\alpha(d-1) + 1) -d],
\end{equation} 
the display \eqref{count_prob} is summable in $n$.  In particular, when $\delta>[\alpha(d-1)+1]^{-1}$, part of our assumptions, a large but fixed $K$ can be chosen so that \eqref{delta_eq} holds.

 Hence, by Borel-Cantelli lemma, on the intersection of a probability $1$ set and $A_1$ say, we recover the claim for all large $n$ 
 that the path visits at most $K$ points between the first and
 last visit to a visited box.
\end{proof}

\subsubsection{Convergence of Optimal Paths}
\label{convergence_subsect}

We now consider the behavior of the optimal paths, in analogy to Theorem \ref{theorem2}, for the energy $H_n$.

\begin{prop}
\label{p=1_convergethm}
Consider the assumptions for the second parts of Theorems \ref{p=1approxthm} and \ref{linear_p=1approxthm}.  Consider a discrete path $\w^{(n)} \in \argmin H_n$ and its linear interpolation, $l_{\w^{(n)}}$.  Then, with respect to realizations in a probability $1$ subset of $A_1\cap A_2$, for any subsequence of $\{l_{\w^{(n)}}\}$, and correspondingly of $\{\w^{(n)}\}$, there is a further subsequence of the linear paths which converges uniformly to a limit path $\gamma\in \argmin F$, and of the discrete paths in the Hausdorff sense to $S_\gamma$.

If $F$ has a unique (up to reparametrization) minimizer $\gamma$, then the whole sequence $\{\w^{(n)}\}$ converges, $\lim_{n\rightarrow\infty}d_{haus}(\w^{(n)},S_\gamma) = 0$.

Consider now a path $l_{\v^{(n)}}\in \argmin L_n$.  The same conclusions holds for $\{\v^{(n)}\}$ in place of $\{\w^{(n)}\}$.
\end{prop}

\begin{proof}  Consider first $\w^{(n)}\in \argmin H_n$.
  By the compactness criterion, Proposition
  \ref{gcompactness}, almost surely, any subsequence of the paths $\{l_{\w^{(n)}}\}$ has a further subsequence $\{l_{\w^{(n_k)}}\}$ converging uniformly
  to a limit $\gamma$. By the `liminf' Lemma
  \ref{liminf},
  $F(\gamma) \leq \liminf_{k \to \infty} F(l_{\w^{(n_k)}})$.  
  
  The same argument and conclusion holds with $l_{\v^{(n)}}\in \argmin L_n$ and $\v^{(n)}$ in place of $l_{\w^{(n)}}$ and $\w^{(n)}$.  
  
  We now show that $\gamma\in \argmin F$.  With respect to optimal $L_n$ paths, as $F(l_{\v^{(n)}})=\min L_n$, and $\min L_n \rightarrow \min F$ a.s. by Proposition \ref{p=1_minapproxthm}, we obtain $F(\gamma)\leq \min F$, and so the desired conclusion.
  
  For $H_n$ optimal paths, 
  we recall an argument in the proof of Theorem \ref{p>1approxthm}.  Using only the `standing assumptions' (allowing $p= 1$) and (Lip), we derived \eqref{fgcompare}, namely, for $\u\in V_n(a,b)$, that
	$|L_n(l_\u) - H_n(\u)|\leq cm_1^{-1}\epsilon_n\min(L_n(l_\u),H_n(\u))$ where $c$ is the constant in (Lip).  Then, as a consequence of Proposition \ref{minprop}, we saw in \eqref{samelimit0} that
	$\min L_n \leq L_n(l_{\w^{(n)}})\leq (cm_1^{-1}\epsilon_n +1)^2\min L_n$.  Since, by Proposition \ref{p=1_minapproxthm},  $\min L_n\rightarrow \min F$ a.s., we conclude that $\gamma\in \argmin F$.

Finally, we remark that the Hausdorff convergences are argued as in the proof of Corollary \ref{hausdorffcor}.
\end{proof}

\section{Appendix}
\label{appendix}
Here we collect some results which we had previously assumed.

\subsection{Nearest-Neighbor Rate}

\begin{prop} \label{nearestneighbor}
Let $\{X_i\}$ be i.i.d. samples from a probability measure $\nu = \rho(x)\, dx$ on a Lipschitz domain $D$, and let \[R_n = \sup_{y \in D} \min_{1 \leq i \leq n} |X_i - y|. \]
Suppose $\rho(x)$ is uniformly bounded below by a positive constant.  Then, there exists a constant $C$, independent of $n$, such that, for almost all realizations $\{X_i\}$, 
\begin{equation*}
  \limsup_{n \to \infty} \frac{R_n n^{1/d}}{(\log n)^{1/d}} \leq C.
\end{equation*}
In particular, when $\epsilon_n$ satisfies \eqref{rateassumption}, for $a,b\in D$, almost surely for all large $n$, there is a path in $V_n(a,b)$ connecting $a,b$ via the graph $\mathcal{G}_n$. 
\end{prop}

\begin{proof}  We first address the claim with respect to $R_n$.
  Let $B(y,r)$ be the Euclidean ball of radius $r$ centered
  at $y \in D$. Since $D$ is Lipschitz, there is a constant
  $c$ such that $m(B(y,r) \cap D) \geq c m(B(y,r))$ for all small $r>0$, where $m$ denotes
  Lebesgue measure (cf. the discussion about cone conditions in
  Section 4.11 of \cite{adams2003sobolev}). It follows that there is a
  constant $c$ such that $m(B(y,r) \cap D) / m(D) \geq c r^d$ for all
  $y \in D$ and all small $r>0$. Since $\nu$ has density $\rho$ bounded below by a positive
  constant, there exists a constant $c$ such that
  $\nu(B(y,r)) \geq c r^d$ for all $y \in D$ and $0 < r < r_0$, where
  $r_0$ is a sufficiently small constant. Therefore, recalling $\mathcal{X}_n = \{X_1,X_2,\ldots,X_n\}$, we have
  \begin{equation} \label{kiss1}
    \begin{aligned} \mathbb{P}(|B(y,r) \cap
      \mathcal{X}_n | = 0 ) &=
      (1 - \nu(B(y,r)))^n \\
      & \leq (1 - cr^d)^n \  
       \leq e^{-cnr^d}.
    \end{aligned}
    \end{equation}
  Let $\{y_1,\ldots,y_k\} \subset D$ be a collection of points so that
  $\sup_{x \in X} \min_{1 \leq i \leq k} |x-y_i| \leq r$. We may take
  the number of points $k$ to satisfy $k \leq c / r^d$ for some constant $c$ independent of $r$, say, by
  choosing $\{y_i\}$ to be a regular grid, with grid length $\sim r$.

  Let $E_i$ denote the event that $|B(y_i,r) \cap \mathcal{X}_n | = 0$, and consider the event $\{ R_n > 2r \}$ that there exists a $y \in D$ with $\min_{1 \leq i \leq n} |X_i-y| > 2r$.  Then, by a triangle inequality argument, we have $\{R_n > 2r \} \subset \cup_{i=1}^k E_i$. Hence, together with \eqref{kiss1}, we have
  \begin{equation} \label{kiss2}  
    \mathbb{P}(R_n > 2r)  \leq  \sum_{i=1}^k \mathbb{P}(E_i) \ \leq \ \frac{c}{r^d} e^{-cnr^d}.
\end{equation}
  Let $r^d = (3 \log n)/(c n)$. Then, \eqref{kiss2} gives a summable term,
  \[ \mathbb{P}\Big(R_n > 2 \frac{(3 \log n)^{1/d}}{n^{1/d}}\Big) \leq \frac{c^2}{3n^2}. \]
By Borel-Cantelli lemma, $R_n\leq 2 (3 \log n)^{1/d}/ n^{1/d}$ for all large $n$.
    
%

  

  We now show that $G_n(a,b)$ is connected when $\epsilon_n$ satisfies \eqref{rateassumption}.    Let $v_1, v_2$ be any vertices in $\mathcal{X}_n\cup \{a,b\}$, and consider the line $\ell(t) = v_1(1-t) + v_2(t)$ between them for $t\in [0,1]$.  By convexity of $D$, the path $\ell$ is contained in $D$.  Consider points on the path $v_1=\ell(0), \ell(R_n), \ell(2R_n),\ldots, \ell(kR_n), \ell(1)=v_1$, where $k= \lfloor |v_2-v_1|/R_n \rfloor$ so that $|1-kR_n|\leq R_n$.  Each point $y=\ell(jR_n)$ is within Euclidean distance $R_n$ of a point $u_j\in \mathcal{X}_n$ by the `$R_n$-limit' a.s. for all large $n$.  By construction, for $1\leq j\leq k-1$, the Euclidean distance between $u_{j}$ and $u_{j+1}$ is less than sum of the distances, from $u_j$ to $\ell(jR_n)$, from $\ell(jR_n)$ to $\ell((j+1)R_n)$, and from $\ell((j+1)R_n)$ to $u_{j+1}$, which is bounded by $3R_n$.  Similarly, the endpoints $v_1$, $v_2$ are within Euclidean distance $2R_n$ of $u_1$ and $u_k$ respectively.  Since, by \eqref{rateassumption}, $3R_n/\epsilon_n<1$ for all large $n$, the path along vertices $\{v_1, u_1,\ldots, u_k, v_2\}$ belongs to $V_n(a,b)$ and so $v_1$ and $v_2$ are connected in $G_n(a,b)$ a.s. for all large $n$.  
%
\end{proof}

\subsection{Existence of Quasinormal Minimizers}

We discuss a `conservation law' for $F$-minimizing paths, and existence of $F$-minimizing Lipschitz paths, following the treatment in \cite{buttazzo1998one}.

\begin{prop} \label{existence} Consider the integral functional
  $ F(\gamma) = \int_0^1 f(\gamma,\dot{\gamma})\, dt $, where $f$ satisfies (A0)-(A3). 
Then, $F$  attains a minimum on the set $\Omega(a,b)$ of Lipschitz paths from $a$ to $b$. In other words, there exists a $\gamma^* \in \Omega(a,b)$ with $F(\gamma^*) = \inf_{\gamma \in \Omega(a,b)} F(\gamma)$.

  In the case that $p=1$, there exists a $\gamma \in \argmin F$ and constants $c, c_1,c_2$ such that 
  \begin{equation} \label{conservelaw}
      f(\gamma(t),\dot{\gamma}(t)) = c \hspace{0.5cm} \text{a.e. } t \in [0,1],
  \end{equation}
  and 
  \begin{equation} \label{derivbounds}
    c_1 \leq |\dot{\gamma}(t)| \leq c_2 \hspace{0.5cm} \text{a.e. } t \in [0,1].
  \end{equation}
  If $p > 1$, then there are constants $c,c_1,c_2$ such that \eqref{conservelaw} and \eqref{derivbounds} hold for any $\gamma \in \argmin F$.
\end{prop}

\begin{proof} We first give an argument in the case where $p > 1$.
  Note, by assumption, the integrand
  $f$ is continuous and $C^1$ on
  $D\times (\mathbb{R}^d\setminus\{0\})$, convex and $p$-homogenous in
  the second argument, and satisfies \eqref{fbounds}.  As $p>1$, $f$
  may be extended continuously to a $C^1$ function on
  $D\times \mathbb{R}^d$.
  
  When the domain of $F$ is extended to all Sobolev paths
  $\gamma \in W^{1,p}([0,1]; D)$ with $\gamma(0) = a, \gamma(1) = b$,
  the existence of a minimizer follows from Remark 2 of Section 3.2 in
  \cite{buttazzo1998one}. In particular, the continuity and convexity
  assumptions (A0) and (A1) imply that $F$ is lower-semicontinuous
  with respect to weak convergence of Sobolev functions. The
  existence of a minimizer then follows from a standard compactness
  argument.
  
  Let $\gamma$ denote such a Sobolev minimizer. Consider now an
  `inner variation' $\omega(t,\epsilon) = \gamma(\xi(t,\epsilon))$
  of $\gamma$, where $\xi$ is $C^1$ on
  $[0,1]\times (-\epsilon_0,\epsilon_0)$ for some $\epsilon_0 > 0$ and
  $\xi(\cdot,\epsilon)$ is a $C^1$ diffeomorphism of the interval
  $[0,1]$ to itself, with $\gamma(t,0) = t$. It may be shown (see the
  discussion on pages 19-21, Proposition 1.16 and Remark 3 in Section
  1.1 of \cite{buttazzo1998one}) that the optimality condition
  $\frac{d}{d \epsilon} F(\omega(\cdot,\epsilon)) |_{\epsilon = 0}$
  over the class of inner variations, together with Euler's identity
  for homogenous functions, $v\cdot\nabla_v f(x,v) = pf(x,v)$, together imply
  that
  \begin{equation} \label{conserveintermediate}
    (p-1)f(\gamma(t), \dot\gamma(t)) = c \ \ \ {\rm a.e.} \ t,
    \end{equation}
    for some constant $c$.

    Finally, by assumption (A3) on $f$, it follows that $c > 0$ and there exist constants $c_1,c_2 > 0$ with $c_1 \leq |\dot{\gamma}(t)| \leq c_2$ for almost every $t$.  In particular, $\gamma\in \Omega(a,b)$, and the proposition is proved for $p > 1$.

  The argument for the $p = 1$ case is complicated by a lack of
  compactness with respect to weak convergence in the Sobolev space
  $W^{1,1}([0,1];D)$, as well as difficulty in establishing an analogue of \eqref{conserveintermediate}.  By a more involved argument, relating optimizers of $F$
  to optimizers of the quadratic functional
  $Q(\gamma) \coloneqq \int_0^1 f^2(\gamma,\dot\gamma)\, dt$, the
  existence of a Lipschitz path $\gamma \in \argmin F$ satisfying
  \eqref{conservelaw} is established in Theorem 1 of
  \cite{hildebrandt} (see also Theorem 5.22 of
  \cite{buttazzo1998one} which gives an alternative argument). From this and assumption (A3), inequality \eqref{derivbounds} follows.
\end{proof}

 \vskip .2cm
\noindent {\bfseries Acknowledgement.}   
This work was partially supported by ARO W911NF-14-1-0179.

\end{document}